\documentclass[letterpaper,11pt]{article}
\usepackage[english]{babel}
\usepackage[ansinew]{inputenc}
\usepackage{fontenc}
\usepackage{amsfonts}
\usepackage{amsmath}
\usepackage{amsthm}
\usepackage{enumerate}
\usepackage{textcomp}
\usepackage{geometry}
\usepackage{mathrsfs}
\usepackage{natbib}
\usepackage[titletoc]{appendix}

\usepackage[colorlinks=true,linkcolor=red,citecolor=blue]{hyperref}

\newcommand{\n}{\noindent}
\theoremstyle{plain} 
\newtheorem{theorem}{Theorem}[section]

\newtheorem{proposition}{Proposition}[section]
\newtheorem{lemma}{Lemma}[section]
\newtheorem{definition}{Definition}[section]

\newtheorem{remark}{Remark}[section]

\begin{document}

\title{Modeling and study a general vacation queuing system with impatience customers}
\author{Assia Boumahdaf \footnote{E-mail: assia.boumahdaf@gmail.com}}
\maketitle

\begin{abstract}
In this paper we model and study a general vacation queueing model with impatient customers. We first propose a sufficient condition for the existence of the stationary workload process. We then give an integral equation for the independent and identically distributed case. This integral equation is solved when customers arrive according to a Poisson point process. A relationship between the tail of the waiting time distribution and the tail of service distribution is also given.
\end{abstract}

\section{Introduction}
\label{introduction}

We introduce and study a new model of a vacation queueing model with impatient customers. 
Such models may be used to describe packet switching networks involving real-time applications. 
These applications are characterized by a data transmission time within a very short time.
Moreover, such systems transmit packets flows that come from different sources and are multiplexed in a common queue, resulting in a delay transmission (the waiting time). 
This delay transmission, is one of the main elements of quality of service (QoS) in such systems. 
When we are concerned with the performance of one particular stream, the other flows (waiting in the same queue) may be regarded as secondary flows. 
So, the server may be regarded as a server on vacation whenever it transmits these secondary flows. 
We will focus specifically on modelling the workload and studying the performance of a vacation queueing model with impatient customers. 
We will model the workload process by a stochastic recursive sequence (SRS).
The workload at time $n$ is defined as the amount of work remaining to be done by the server at time $n$. 
This equation is described by the recursion $W_{n+1} = h(W_n,\xi_n)$, where the random variable $W_{n}$ denotes the workload viewed by the $n$th customer, $\{\xi_n\}$ is some stationary random variable, and $h(.)$ some deterministic and known function.   
Such a model, has never been considered before in the literature. 
Recursions of this type which may see as a generalisation of Lindley's recursion \cite{Lindley52} is one of the fundamental and most well-studied equations in queuing theory. 
For earlier works on such stochastic recursions, see for example \cite{BaccelliBremaud}, \cite{BorovkovFoss1992} or \cite{Brandtbook}. 
The first important issue that we consider is the stability of the system. 
We will consider the problem of the existence and uniqueness of a stationary version of the workload process by using the theory of renovating events of Borovkov (\cite{Borovkov78}, \cite{Borovkovbook} ) in the general case where the underlying processes are stationary. 
Such theory have been used in the context of queues with impatience by \cite{Moyal2010}. 
Under stability conditions and independence assumptions, an integral equation for the stationary probability density function (p.d.f.) of the workload is derived. 
Using the Laplace transform approach a closed-form expression for the steady-state limiting distribution of the workload given when the customers arrive according to a Poisson process and the patience times are exponentially distributed. 
If one cannot determine analytically the limiting distribution function of the waiting time it is of interest of studying the tail behaviour using informations from the given distributions.

Queueing models with impatience and vacation times have been largely studied. 
The phenomena of impatience is also referred in the applied probability literature as limited waiting (or sojourn) times. Most of these works using integral equation approach, focus on the $M/G/1$ queueing models with a general vacation distribution and constant deadline. 
The study of these queueing systems have been initiated by \cite{vanderDuynSchouten78}.
The author derives the joint stationary distribution of the workload and the state of the server (available for service or on vacations). 
\cite{TakineHasegawa90} have considered two $M/G/1$ with balking customers and deterministic deadline on the waiting and sojourn times. 
They obtain an integral equation for the steady state probability distribution function of the waiting times and the sojourn times. 
They expressed them in terms of steady state probability distribution function of the $M/G/1$ queue with vacations without deadline.
Recently, \cite{Katayama2011} has investigated the $ M/G/1 $ queue with multiple and single vacations, sojourn time limits and balking behaviour. 
Explicit solutions for the stationary virtual waiting time distribution are derived under various assumptions on the service time distribution.
\cite{Katayama2012} derives recursive equations in the case of deterministic service times for the steady-state distributions of the virtual waiting times in a $ M/G/1 $ queue with multiple and single vacations, sojourn time limits and balking behaviour. 

This paper is organised as follows. 
In Section~\ref{section 1}, we describe the model for the workload process. 
In Section~\ref{section 2}, we provide a sufficient condition for the existence and the uniqueness of the stationary workload. 
In Section~\ref{section 3}, an integral equation is established when customers arrive according to a Poisson process and the patience time is exponentially distributed.
In Section~\ref{section 4}, we derive an integral equation satisfied by the steady-state probability distribution function and a close-form expression of its Laplace transform is proposed. 
In Section~\ref{section 5}, we investigate the tail of the workload distribution.

\section{Description of the model}
\label{section 1}

\subsection{Description of the probability space}
\label{subscetion:description probability space}

Let $\left( \Omega, \mathscr{A}, \mathbb{P} \right)$ be a probability space. All random variables under consideration will be defined on $\left(\Omega, \mathscr{A}, \mathbb{P} \right)$. 
Let $\theta$ be a measurable map from  $\left(\Omega, \mathscr{A}\right)$ into itself such that $\mathbb{P}$ is $\theta$-invariant, that is $\mathbb{P}\left( \theta^{-1}(A) \right) = \mathbb{P}\left( A\right)$ for all $A \in \mathscr{A}$, where $\theta^{-1}$ denotes the measurable inverse of $\theta$. We assume that $\mathbb{P}$ is $\theta$-ergodic, \textit{i.e.}, for all $A \in \mathscr{A}$ that are $\theta$-invariant, that is $\theta^{-1}(A) = A$ then $\mathbb{P}(A) = 0$ or $1$. The iterates of the mapping $\theta$ are defined by composition, where $\theta^{0}$ is the identity map, and for $n \in \mathbb{N}$,
\begin{equation*}
\theta^{n} = \underbrace{\theta \circ \theta \circ \ldots \circ \theta}_{n \mbox{ times }}  \qquad \mbox{and} \qquad \theta^{-n} = \underbrace{\theta^{-1} \circ \theta^{-1} \circ \ldots \circ \theta^{-1}}_{n \mbox{ times }}.
\end{equation*}

We now proceed to describe our queueing system. We consider first-in-first-out queueing system with infinite waiting room in which customers enter at random times $\ldots, T_{-1} < T_0 < T_1 \ldots$. We denote for all $n \in \mathbb{Z}$, $ \tau_n = T_{n} - T_{n-1}$ the inter-arrival time between the $n$th customer and the $(n-1)$th one.
Customers are denoted in the all the sequel $\ldots,  C_{-1}, C_{0}, C_{1}, \ldots$. They require service times  $\left( \sigma_{n}, \, n \in \mathbb{Z} \right)$. Furthermore, they are assumed \textit{impatient}, that is, they require to reach the server before a given deadline $\left( D_n, \, n \in \mathbb{Z} \right)$ otherwise, he lives the system and never returns. More precisely, if a customer reaches the server before his deadline he remains until completion.
The workload upon arrival of customer $C_n$, denoted by $W_n$ is assumed to be known. Hence $C_n$ enters the system if and only if the $W_n < D_n$. If not, $C_n$ does not enter and never returns. Whenever a customer is patient $W_n$ correspond to his waiting time.
We assume furthermore, that the server is subject to service interruptions $\left( V_n, n \in \mathbb{Z}\right)$, called \textit{vacations}. When the system becomes idle, the server takes one vacation. If on return from  a vacation the system is not empty, the server begins service, otherwise it waits for the next arrival. 
\vspace{0.2cm}

We shall further assume that $\left( \Omega, \mathscr{A} \right)$ is the canonical space of the sequence \[ \Big( \left(T_n, \sigma_n, D_n, V_n \right), \, n \in \mathbb{Z} \Big). \] 
Specifically, we take $\Omega = (\mathbb{R}_{+}^4)^{\mathbb{Z}}$ and $\mathscr{A} = \mathscr{B}((\mathbb{R}_{+}^4)^{\mathbb{Z}})$, where $\mathscr{B}((\mathbb{R}_{+}^4)^{\mathbb{Z}})$ is the Borel sigma-field on $(\mathbb{R}_{+}^4)^{\mathbb{Z}}$.
If $\omega \in \Omega$ then $\omega := \Big( (\omega_k^1,\omega_k^2, \omega_k^3,\omega_k^4 ), \, k \in \mathbb{Z} \Big)$ and the shift operator is thus defined on $\Omega$ by
\begin{equation*}
 \theta(\omega) = \Big( (\omega_{k+1}^1,\omega_{k+1}^2, \omega_{k+1}^3,\omega_{k+1}^4 ), \, k \in \mathbb{Z} \Big).
\end{equation*}
For $\omega \in \Omega$, we define on $\left(\Omega, \mathscr{A} \right)$ the generic random variables
\begin{eqnarray*}
\tau(\omega)   &=& \tau\left( \omega_k, \, k \in \mathbb{Z} \right) := \omega_{0}^{1},\\ 
\sigma(\omega) &=& \tau\left( \omega_k, \, k \in \mathbb{Z} \right) := \omega_{0}^{2},\\
D(\omega)      &=& \tau\left( \omega_k, \, k \in \mathbb{Z} \right) := \omega_{0}^{3},\\
V(\omega)      &=& \tau\left( \omega_k, \, k \in \mathbb{Z} \right) := \omega_{0}^{4}.
\end{eqnarray*}

\n The random variables $\tau$, $\sigma$, $D$ and $V$ are interpreted as, respectively, the inter-arrival time starting from $0$, the service time of the customer $C_0$ arrived at time $0$, the deadline of $C_0$ and the vacation time of index $0$. Hence, for all $n \in \mathbb{Z}$, $\tau \circ \theta^n$, $\sigma \circ \theta^n$, $D \circ \theta^n$ and $V \circ \theta^n$ represent, respectively, the inter-arrival time between $C_n$ and $C_{n+1}$, the service time and the deadline of $C_n$, and the $n$th vacation period. The mapping $\theta$ passes from a customer to the following one in order of arrivals. Then, it follows from our readily assumptions that the sequence $\Big( (\tau_n, \sigma_n, D_n , V_n), \, n \in  \mathbb{Z}\Big)$ is stationary and ergodic.
We assume moreover that the random variables $\tau$, $\sigma$, $D$ and $V$ are integrable and that $\tau > 0$ $\mathbb{P}$-a.s..  We fix the time origin $T_0 = 0$ at the arrival time of $C_0$. The quadruplet $\left( \Omega, \mathscr{A},\mathbb{P},\theta \right)$ is usually called the \textit{Palm space} of the sequence $\Big( \left(T_n, \sigma_n, D_n, V_n \right), \, n \in \mathbb{Z} \Big)$.

\subsection{The workload sequence}
\label{subsection:workload process}

Let $W_{t}$, $t \in \mathbb{R}$ be the workload at time $t$, that is the amount of work remaining to be done by the server at time $t$. By convention, the \textit{workload process} $\{ W_t, t \in \mathbb{R} \}$ is taken right-continuous with left limit $W_{t^-}$, with $W_{0^-} = 0$. 
We define the \textit{workload sequence} $\left( W_n, n \in \mathbb{Z} \right)$ by $W_n = W_{T_{n}^-}$, for all $n \in \mathbb{Z}$.
The value of $W_n$ taken up to time $T_n$ represents the time that the customer $C_n$ would have to wait to reach the server. Under the First In First Out (FIFO) discipline, patient customers are those whose the deadline exceeds the workload upon their arrivals, \textit{i.e.}, the customer $C_n$ enters the system if and only if $W_n < D_n$. In literature of queueing system, we say that the customer balks upon arrival.
In the sequel, we will use the notation $\left( X_n^{X}, n \in \mathbb{Z} \right)$ to indicate that the initial condition of the sequence $\left( X_n, n \in \mathbb{Z} \right)$ is given by $X_0 = X$ $\mathbb{P}$-a.s..
The first issue, we are concerned with, is the representation of the workload sequence $\left( W_n^W, n \in \mathbb{Z}\right)$ by a \textit{stochastic recursive sequence} (SRS) that expresses the workload $W^{W}_{n+1}$ of the customer $C_{n+1}$ in terms of $W^{W}_n$. In the sequel, we omit the exponent if there is no confusion. The representation of the workload sequence by a SRS will allow us to provide stability condition under general assumptions of ergodicity. 
\vspace{0.3cm}

Before defining the model, let us start by introducing some definitions and notations. We define at first the random variable $X_n$ given by Equation~\eqref{chap2-proba:def variable Xn} below, and afterwards, the concept of $(m,n)$-family. In all the sequel, we assume that $W_0 = W$ $\mathbb{P}$-a.s., where $W$ is a non-negative random variable and that the customer $C_0$ reaches the server and remains in the system until completion of his service. Define the random variable $X_{n}$ by
\begin{equation}
\label{chap2-proba:def variable Xn}
  X_{n} := W_{n} + \sigma_{n}\mathbf{1}_{W_{n} < D_{n}} - \tau_{n}, \quad n \in \mathbb{Z}.
\end{equation}

\n As for the standard Lindley recursion we discuss possible configurations depending on the sign of $X_n$. Two possible situations occur: when $X_n > 0$ and when $X_n \leq 0$. 
When the event $\{ X_n > 0\}$ occurs, at time $T_n$ (the arrival time of $C_n$) the system is not idle; whereas on $\{ X_n \leq 0\}$ the system is idle at $T_n$.
\vspace{0.3cm}

$\bullet$ If $X_{n} > 0$, then two cases occur: when the customer $C_n$ is served (\textit{i.e.}, $W_n < D_n$) and when $C_n$ balks upon arrival ($W_n \geq D_n$). If $C_n$ is patient then $X_{n} = W_{n} + \sigma_{n} - \tau_{n}$, and thus the customer $C_{n+1}$ arrives during the sojourn time (the waiting plus the service time) of $C_{n}$. In such a case
\begin{equation*}
 W_{n+1} = W_n + \sigma_n - \tau_n.
\end{equation*}

\n Otherwise, if $C_{n}$ is impatient (he abandons the system upon his arrival) then $X_{n} = W_{n} - \tau_{n}$. In this case, $X_{n} \geq 0$ means that if $C_{n}$ had been patient, the customer $C_{n+1}$ would have arrived during the waiting time of $C_{n}$. Whatever the state of the server (in service or in vacation), provided that $C_n$ and $C_{n+1}$ arrive during the same state, the workload of the customer $C_{n+1}$ is given by
\begin{equation*}
 W_{n+1} = W_n  - \tau_n.
\end{equation*}

$\bullet$ If $X_{n} \leq 0$ the situation is slightly different because the server can take vacations and thus the workload of $C_{n+1}$ may be expressed in terms of these vacations. Indeed, when $X_n \leq 0$, this may mean two things: either the server begins one vacation or the server returns from its vacation. 
Consider at first, the case where $C_n$ is patient. Then, the customer $C_{n+1}$ arrives after the departure of $C_{n}$. At the departure time of $C_n$, the system is idle and the server takes one vacation $V_n$. If $C_{n+1}$ arrives during the vacation period $V_n$ then its workload is expressed as
\begin{equation*}
 W_{n+1} = W_n + \sigma_n + V_n - \tau_n.
\end{equation*}
Whereas, if after returning from its vacation, the server finds an idle system, the workload of $C_{n+1}$ is $W_{n+1} = 0$. Combining both results, we have 
\begin{equation*}
W_{n+1} = \left[ W_n + \sigma_n + V_n - \tau_n \right]^+.
\end{equation*}

When $X_n \leq 0$ and $W_n > D_n$, the workload of $C_{n+1}$ may be expressed in terms of vacation or not. Indeed, $C_{n+1}$ may arrive during a vacation period or arrive after this period.
In order to distinguish both cases, we need to introduce some definitions. In the sequel, we will introduce the concept of \textit{$(m,n)$-family of customers} and we need also to know how many number of customers are served with a $(m,n)$-family (where $m$ is a random variable). This number will enable us to identify the cases where $W_{n+1}$ is expressed in terms of vacation or not. 
\vspace{0.3cm}

Let us now define the concept of a $(m,n)$-family. Such a family will be denoted by $\mathcal{F}^{n}_{m}$, $m > 0$, $n \geq m$ and it consists of all customers $C_{m}, \ldots , C_{n}$ satisfying $X_{m-1} < 0, X_{m} >0, X_{m+1} >0, \ldots , X_{n-1} > 0, X_{n} <0$, \textit{i.e.},
\begin{equation*}
  \mathcal{F}^{n}_{m} = \left\lbrace C_{m}, \ldots , C_{n} \right\rbrace .
\end{equation*}

\n
When $m = 0$ the $(0,n)$-family is the family of customers satisfying $X_{0} >0 ,\ldots X_{n} < 0$, i.e.
\begin{equation*}
\mathcal{F}_{0}^{n} = \{ C_{0}, \ldots , C_{n} \}.
\end{equation*}  
Such a family may be composed only by one customer. For example, $\mathcal{F}_{n}^{n} = \{C_n \}$ contains only  $C_n$. In such a case we have $X_{n-1} < 0$, and $X_n < 0$. Whenever $m=n=0$, the family $\mathcal{F}^{0}_{0} = \{  C_0 \}$ satisfies only one condition $X_{0} < 0$. In conclusion, when $0 \leq m < n$, a family of customers must begin by customer $C_{m}$, such that $X_{m} > 0$, and must end by a customer $C_{n}$ satisfying $X_n < 0$. 

We now focus on the number of served customers within a $(m,n)$-family, $\mathcal{F}_{m}^{n}$. This number will be of major importance to define the workload $W_{n+1}$ when $X_n \leq 0$ and $W_n > D_n$. 
Assume that $X_{n} \leq 0$ and $C_n$ is impatient, \textit{i.e.}, $W_n > D_n$, then the $(m,n)$-family is well defined ($m$ is a realization of some random variable $M_n$. We shall explicitly define it right after). The number of served customers within $\mathcal{F}_{m}^{n}$ will be denoted by $N_n$ and is defined by
\begin{equation*}
N_n := \sum_{k=m}^{n} \mathbf{1}_{W_k \leq D_k} \in \{0, \ldots , n\}.
\end{equation*}

$\bullet$ If $N_n > 0$, then there at least one patient customer within $\mathcal{F}_{m}^{n}$. Consequently, at the instant where the event $\{ X_n \leq 0 , W_n > D_n\}$ occurs, the system is idle and the server takes one vacation at the departure time of the last served customer. The workload of $C_{n+1}$ is thus given by
\begin{equation*}
W_{n+1} = \left[ W_n + V_n - \tau_n \right]^+.
\end{equation*}

$\bullet$ Whereas, on the event $\{ X_n \leq 0 , W_n > D_n, N_n = 0 \}$ the system is idle but all customers within $\mathcal{F}_{m}^{n}$ are impatient. In this case, we have 
\begin{equation*}
W_{n+1} = \left[ W_n - \tau_n \right]^+ = 0.
\end{equation*}

Back to the definition of the random variable $M_n$. This random variable characterizes the the first customer of a family and is defined by
\begin{equation}
\label{chap2-proba:definition M_n}
M_n :=  \sup\left\lbrace 1 \leq r \leq n\,:\, X_{r-1} \leq 0 \mbox{ and } X_n \geq 0 \right\rbrace  \mbox{ on } \{ X_n \leq 0 \},
\end{equation}

with the convention $\sup \emptyset = - \infty$.
\vspace{0.3cm}

 The above discussion is summarized in the following Proposition.

\begin{proposition}
\label{chap2-proba:proposition:defnition SRS}
For a $G/G/1/V_s+G$ queue, the workload sequence $\{ W_n, n \in \mathbb{N} \}$ satisfies the Lindley type recursion, for all $n \geq 0$ 
\begin{equation}
\label{Lindley rec}
W_{0} = W\;\; \mathbb{P}-a.s. \qquad \textrm{and} \qquad     W_{n+1} = \left\lbrace \begin{array}{llll}
                   X_{n} \hspace{3.5cm} \mbox{if} \;\; X_{n} > 0, \vspace{0.3cm}\\
                   \left[ X_{n} + V_{n}\mathbf{1}_{N_{n}>0} \right]^{+} \hspace{1cm} \mbox{if} \;\;\ X_{n} \leq 0 ,  
                        \end{array}
          \right.      
\end{equation}
where the initial condition $W_{0} = W$ $\mathbb{P}$-a.s. is a non-negative random variable, and 
\begin{equation}
\label{chap2-proba:definition N_n}
N_n := \sum_{k=M_n}^{n} \mathbf{1}_{W_k \leq D_k},
\end{equation}
with $M_n$ defined by \eqref{chap2-proba:definition M_n}.
\end{proposition}

\begin{remark}
The Model~\ref{Lindley rec} contains one event of probability one which is $\{ X_n \leq 0,  W_{n} < D_n, N_{n} > 0\} $ and one of probability zero, $\{ W_n \leq D_n, N_n = 0 \}$.
\end{remark}

\n The model~\eqref{Lindley rec} may be written more explicitly as 
\begin{equation}
\label{Lindley rec2}
W_{n+1} = \left\lbrace \begin{array}{llllll}
                   W_{n} + \sigma_{n} - \tau_{n} \hspace{3cm} \mbox{ if} \;\;  W_{n} + \sigma_{n}- \tau_{n} > 0, \, W_n \leq D_n, \vspace{0.3cm}\\
                    W_{n} - \tau_{n} \hspace{4cm} \mbox{if} \;\;  W_{n} - \tau_{n} > 0, \, W_n > D_n, \vspace{0.3cm}\\
                    \left[  W_{n} + \sigma_{n}+ V_{n} - \tau_{n} \right]^{+} \hspace{1.7cm} \mbox{if} \;\;  W_{n} + \sigma_{n} - \tau_{n} \leq 0 ,\, W_n \leq D_n,\vspace{0.3cm}\\
                   \left[  W_{n} + V_{n}- \tau_{n} \right]^{+} \hspace{2.6cm} \mbox{if} \;\;  W_{n} - \tau_{n} \leq 0 ,\, W_n > D_n, N_n  > 0,  \vspace{0.3cm}\\
                    \left[  W_{n} - \tau_{n} \right]^{+} = 0 \hspace{2.8cm}\mbox{if} \;\;  W_{n} - \tau_{n} \leq 0 ,\, W_n > D_n, N_n  = 0.
                        \end{array}
          \right.      
\end{equation}

In the standard queueing model with impatience, the workload process $\{ W_n, n \in \mathbb{N} \}$ is driven by a Stochastic Recursive Sequence (SRS) of the form
\begin{equation*}
 W_{0} = W \qquad \mbox{ and } \quad W_{n+1} = h(W_n,\xi_n),
\end{equation*}
where
\begin{equation}
\label{chap2-proba:def xi}
\xi_n := (\tau_n ,\sigma_n, D_n) \in \mathbb{R}_{+}^3.
\end{equation}

\n The model may be also described by a recursive equation but it has the particularity to depend of all the past through the random variables $N_n$ and $M_n$. We then define
\begin{eqnarray*}
W_{n,0}   &:=& \left( W_n, W_{n-1}, \ldots ,  W_0 \right)\\
\xi_{n,0} &:=& \left( \xi_n, \xi_{n-1},\ldots , \xi_0 \right),
\end{eqnarray*}
where $\xi_n := \left( \tau_n, \sigma_n, V_n, D_n \right) \in \mathbb{R}_{+}^4$. Hence, the workload sequence~\eqref{Lindley rec2} may be represented by the following SRS
\begin{equation}
\label{Lindley rec3}
W_0 = W ,\qquad W_{n+1} :=  \left\lbrace \begin{array}{lll}
       				h_1\left( W_{n}, \xi_{n} \right) \qquad \mbox{ on } \{ X_n >0\} \cup \{ X_n\leq 0, W_n \leq D_n \}, \vspace{0.3cm}\\
       				h_2\left( W_{n,0}, \xi_{n,0} \right) \quad \mbox{ on } \{ X_n \leq 0, W_n > D_n \},
                 \end{array}\right.
 \end{equation}
where $h_1$ and $h_2$ are some specified measurable functions and $\left( \xi_n , n \geq 0 \right)$ is a stationary and ergodic sequence.

\section{Stability}
\label{section 2}

This section is devoted to the study of the stability of the system, \textit{i.e.}, the existence of an equilibrium state. The evolution of this process is described by the recursive sequence~\eqref{Lindley rec3}, where $h$ is not monotonic in the state variable and depends on all the past of the $W_n$ and $\xi_n$. Thus, the technique developed by Loynes \cite{Loynes1962} for finding the stationary version of $W_n$ is not efficient. An alternative to prove the existence and the uniqueness of a stationary regime is the renovating events method introduced by Borovkov in \cite{Borovkov78}. Such a method was applied by Moyal \cite{Moyal2010} for a standard queueing model with impatience. The author provides a sufficient condition for the existence and the uniqueness of the stationary workload. Here again, this technique is not directly applicable to our model~\eqref{Lindley rec3} since it depends on all the past. To overcome this difficulty, we will introduce a recursive sequences $\left( Z_{n}^{W}, n \in \mathbb{N} \right)$ which corresponds to the workload sequence where we have removed all the events depending on the past. In this way, the sequence $\left( Z_n^{W}, n \geq 0 \right)$ will be of the form $Z_{n+1} = \varphi(Z_n, \xi_n)$, where the measurable function $\varphi$ is not monotonic in the state variable but will dependent on the past only at time $T_n$, and where $\xi$ is defined by \eqref{chap2-proba:def xi}. This strategy will allow us to apply the Borovokov theory to $\left( Z_n^{W}, n \geq 0 \right)$ and to provide a partial stability result for the sequence $\left( W_n^{W}, n \ge 0 \right)$, in the following sense.

\begin{definition}
The sequence $\left( W_n , n \geq 0\right)$ is partially stable if there exits a function $f$ defined on $\mathbb{N}$ with values in $\mathbb{N}$ strictly increasing such that
\begin{equation}
W_{f(n)} = Z_n,
\end{equation}
where $Z_n$ is stable sequence.
\end{definition}

We now define the new sequence $\left( Z_n^{W}, n \geq 0 \right)$ by
\begin{align}
\label{Lindley rec Z}
Z_0 = W, \qquad  Z_{n+1} &= \left\lbrace \begin{array}{lll}
                   Z_{n} + \sigma_{n} - \tau_{n} \hspace{3cm} \mbox{if} \;\;  Z_{n} + \sigma_{n}- \tau_{n} > 0, \, Z_n \leq D_n \vspace{0.3cm}\\
                    Z_{n} - \tau_{n} \hspace{4cm} \mbox{if} \;\;  Z_{n} - \tau_{n} > 0, \, Z_n > D_n \vspace{0.3cm} \nonumber\\
                    \left[  Z_{n} + \sigma_{n}+ V_{n} - \tau_{n} \right]^{+} \hspace{1.7cm} \mbox{if} \;\;  Z_{n} + \sigma_{n} - \tau_{n} \leq 0 ,\, Z_n \leq D_n.
                        \end{array} \right.   \\
                        &   :=\varphi(Z_n, \xi_n),
\end{align}
where $\xi_n$ is defined by \eqref{chap2-proba:def xi}. This workload sequence is stochastically recursive, driven by the measurable map $\varphi$ defined, for all $x \in \mathbb{R}_{+}$, by
\begin{equation*}
 \varphi(x) = \left\lbrace \begin{array}{lll}
             x + \sigma \mathbf{1}_{x \leq D} - \tau \hspace{1cm} \mbox{ if } x + \sigma \mathbf{1}_{x \leq D} - \tau >0\vspace{0.2cm}\\
         \left[ x + \sigma + V - \tau \right]^{+} \hspace{0.8cm} \mbox{if} \;\;  x + \sigma - \tau \leq 0 ,\, x \leq D. \end{array} \right.
\end{equation*}
The map $\varphi$ is not monotonic in the state variable, hence Loynes's method cannot be applied. We use then the renovating events to provide a sufficient condition for the existence and uniqueness of a solution to the equation
\begin{equation}
\label{chap2-proba:stationary equation Z}
Z \circ \theta = \varphi(Z, \xi) \quad \mathbb{P}-a.s..
\end{equation}

\n To this end, let us introduce the second sequence $\left( Y_{n}^Y, n \in \mathbb{N} \right)$ defined by
\begin{equation}
\label{chap2-proba:SRS Yn}
Y_{0} = Y \qquad \mbox{ and } \qquad Y_{n+1} = \left[ \max( Y_{n} + V_n,\sigma_n + D_n + V_{n}) - \tau_n  \right]^{+}.
\end{equation}

\n This recursive equation may be written as follows
\begin{equation*}
Y_{n+1} = \psi(Y_{n},\xi_n),
\end{equation*}
where $\xi$ is defined by \eqref{chap2-proba:def xi} and
\begin{equation*}
\psi(x) = \left[ \max(x + V, \sigma + V+D)-\tau \right]^{+}.
\end{equation*}

\n Hence, the SRS $\left( Y_n^{Y}, n \geq 0 \right)$ is stationary if and only if workload $Y$ solves the equation
\begin{equation}
\label{chap2-proba:stationary equation Y}
Y \circ \theta =  \psi(Y, \xi),
\end{equation}
where $\xi$ is defined by \eqref{chap2-proba:def xi}. The sequence $\left( Y_{n}^Y, n \in \mathbb{N} \right)$ has the advantage to be non-decreasing in the state variable, so that Loynes's construction may be applied. In the sequel, we will construct a stationary workload for $\left( Y_{n}^Y, n \in \mathbb{N} \right)$ which enable us to define a sequence of renovating events for $\left( Z_{n}^W, n \in \mathbb{N} \right)$ and to provide a sufficient condition for the existence and uniqueness of a solution to \eqref{chap2-proba:stationary equation Z}.

\begin{lemma}
\label{chap2-proba:lemmma stability of Yn}
If $\mathbb{E}\left( V-\tau \right) < \infty$, then there exists a unique $\mathbb{P}$-a.s. finite solution of \eqref{chap2-proba:stationary equation Y} given by
\begin{equation}
\label{chap2-proba:stationary solution for Yn}
 M_{\infty} := \left[ \sup_{j > 0}\left( \sigma_{-j} + D_{-j} + \sum_{i=1}^{j} V_{-i} - \sum_{i=1}^{j} \tau_{-i} \right) \right]^{+}.
\end{equation}

\end{lemma}

\n The following lemma shows that the workload $Z^{W}_n$ is majorized by $Y^{Y}_n$ provided that the initial conditions $W$ and $Y$ satisfy $W \leq Y$.

\begin{lemma}
\label{chap2-proba:lemma:comparaison Z_n < Y_n}
If $W \leq Y$, the workload processes $ Z_n^{W}$ and $ Y_n^{Y}$ satisfy the following relation, for all $n \in \mathbb{N}$,
\begin{equation}
\label{Zn < Yn}
    Z_n^{W} \leq Y_n^{Y}, \quad a.s..
\end{equation}
\end{lemma}

\n We have the following result.
\begin{theorem} 
\label{chap2-proba:theorem:stability of Z}
In the $G/G/1 + G$ queue with simple vacation times, if $\mathbb{E}(V-\tau) < 0$ and
\begin{equation}
\label{stability condition}
  \mathbb{P} \left[ \sup_{j > 0}\left( \sigma_{-j} + D_{-j} + \sum_{i=1}^{j} V_{-j} - \sum_{i=1}^j \tau_{-i} \right) \leq 0 \right] > 0,
\end{equation}
then, there exists a unique stationary sequence $\{ U \circ \theta^{n} \}$, solution of \eqref{Lindley rec Z}, and such that for all $W$, $\{ Z_n^W, n \in \mathbb{N} \}$ converges with strong coupling to $\{ U \circ \theta^n \}$.
\end{theorem}

We can see the sequence $\left( Z_n, n \geq 0 \right)$ as a subsequence of $\left( W_n, n \geq 0 \right)$. Hence, there exists a function $f$, strictly increasing such that
\begin{equation}
Z_n = W_{f(n)} \qquad \mbox{ and }\qquad  \varphi(Z_n, \xi_n) = h(W_{f(n)},\xi_{f(n)}). 
\end{equation}
Under the assumptions $\mathbb{E}(V-\tau) < 0$ and \eqref{stability condition} there is a strong coupling for the sequences $\left( W_{f(n)}, n \geq 0 \right)$ and $\left( U \circ \theta^n \right)$. Moreover, it is well known that a sequence that couples with a stationary sequence $\left( U \circ \theta \right)$ converges in distribution to $U$. This, could lead to some control on the distribution of $W_n$ with respect to that of $U$


\section{Independent and identically distributed case: \texorpdfstring{$GI/GI/V_S/1 + GI$}{Lg}}
\label{section 3}

In this section, we consider the case where all sequences $\{ \tau_n, n \in \mathbb{Z} \}$, $\{ \sigma_n, n \in \mathbb{Z} \}$, $\{ V_n, n \in \mathbb{Z} \}$ and $\{ D_n, n \in \mathbb{Z} \}$ are independent and identically distributed, and independent of each other. 
We assume that $W_n$ admits stationary limit regime, which can be argued using the i.i.d. framework. We will also focus on the stationary integral equation satisfied by the probability density function of the workload. 
\vspace{0.3cm}

Let $F_{n}(x)$, $x \in \mathbb{R}^{+}$ be the distribution function of $W_{n}$. We denote by $f_n$ the derivative of $F_n$ when it exists. Referring to model~\eqref{Lindley rec}, we have the following inequalities. For fixed $x > 0$,

\begin{align*}
\begin{split}
0 \leq W_n \leq x & \Leftrightarrow \left\lbrace W_n + \sigma_n - \tau_n \leq x, W_n \leq D_n, W_n + \sigma_n - \tau_n \geq 0  \right\rbrace \\
                  & \qquad \cup \left\lbrace W_n - \tau_n \leq x, W_n \geq D_n , W_n - \tau_n \geq 0 \right\rbrace \\
                  & \qquad \cup \{W_n +\sigma_n + V_n - \tau_n \leq x , W_n \leq D_n , W_n+ \sigma_n - \tau_n \leq 0\} \\
                  & \qquad \cup \{W_n + V_{n} - \tau_n \leq x, W_n \leq D_n, W_n - \tau_n \leq 0,N_{n+1} >0\}.
\end{split}
\end{align*}

\n \\
In the following we will derive an integral equation for $f_n$, the pdf of $W_n$. The classical method builds upon the above inequalities to derive the integral equation, which expresses $F_{n+1}(.)$ in terms of $F_{n}(0)$ and $f_{n}(.)$. The distribution function of model~\eqref{Lindley rec} is given, for all $x \in \mathbb{R}^{+}$, by

\begin{align*}
\begin{split}
F_{n+1}(x) &= \mathbb{P}\left( W_{n+1} \leq x \right) \\
           &= \mathbb{P}\left(  W_n + \sigma_n - \tau_n \leq x, W_n \leq D_n, W_n + \sigma_n - \tau_n \geq 0  \right)\\
           & \qquad + \mathbb{P}\left( W_n - \tau_n \leq x, W_n \geq D_n , W_n - \tau_n \geq 0\right)\\
           & \qquad + \mathbb{P} \left( W_n +\sigma_n + V_n - \tau_n \leq x , W_n \leq D_n , W_n+ \sigma_n - \tau_n \leq 0 \right)\\
           & \qquad + \mathbb{P}\left( W_n + V_{n} - \tau_n \leq x, W_n \leq D_n, W_n - \tau_n \leq 0,N_{n+1} >0\right) \\
           &:= I(x) + J(x) + K(x) +L(x). 
\end{split}
\end{align*}

\n
We now calculate the four probabilities. Conditioning on $\tau_n$ and then on $W_n$ and using independence, the first probability gives 
\begin{align*}
\begin{split}
I(x) :& = \mathbb{P}\left( W_n + \sigma_n - \tau_n \leq x ,W_n \leq D_n, W_n + \sigma_n - \tau_n \geq 0 \right) \\
      & = \int_{0}^{\infty}dA(t)\int_{0^-}^{x+t}dF_n(u)\mathbb{P}(\sigma_n \leq x+t-u, u \leq D_n, \sigma_n \geq t-u).
\end{split}
\end{align*}

\n
The independence between $\sigma_n$ and $D_n$ yields
\begin{align}
\label{I(x)}
\begin{split}
I(x) & = \int_{0}^{\infty}dA(t)\int_{0}^{x+t}dF_n(u) \mathbb{P}(u \geq D_n)\mathbb{P}(t-u \leq \sigma_n \leq x+t-u)\\
     & = \int_{0}^{\infty}dA(t)\int_{0}^{x+t} \bar{G}(u) dF_n(u)\int_{t-u}^{x+t-u} dB(s).
\end{split}
\end{align}

\n
Using the same arguments as above, for $x > 0$,
\begin{align}
\label{J(x)}
J(x) &:= \mathbb{P}(W_n - \tau_n \leq x, W_n \geq D_n , W_n - \tau_n \geq 0 ) \nonumber \\
     &= \int_{0}^{\infty} dA(t)\int_{t}^{x+t} dF_{n}(u) G(u),
\end{align}

\n and
\begin{align}
\label{K(x)}
K(x) &:= \mathbb{P}(W_n +\sigma_n + V_n - \tau_n \leq x , W_n \leq D_n , W_n+ \sigma_n - \tau_n \leq 0)\nonumber \\
     & = \int_{0}^{\infty}dA(t) \int_{0^-}^{t}dF_n(u) \bar{G}(u)\int_{0}^{t-u}dB(s)V(x+t-u-s).
\end{align}

\n 
The calculation of the last probability slightly differs
\begin{align*}
L(x) &:= \mathbb{P}(W_n + V_{\ell_{n+1}} - \tau_n \leq x, W_n \geq D_n, W_n - \tau_n \leq 0,N_{n+1} >0,0 \leq \ell_{n+1} < n)\\
     & = \int_{0}^{\infty} dA(t) \int_{0}^{t} dF_n(u)\mathbb{P}(V_{\ell_{n+1}} \leq x+t-u, u \geq D_n,N_{n+1} >0,0 \leq \ell_{n+1} < n).
\end{align*}


\n Moreover, we have
\begin{align}
\label{L(x)}
 L(x) &= \int_{0}^{\infty} dA(t) \int_{0}^{t} dF_n(u)\mathbb{P}(V_n \leq x+t-u, u \geq D_n) \nonumber \\
      &= \int_{0}^{\infty} dA(t) \int_{0}^{t} G(u) dF_n(u) V(x+t-u).
\end{align}

\n
Combining Equations~\eqref{I(x)},\eqref{J(x)},\eqref{K(x)} and \eqref{L(x)}, the integral equation for the distribution function $F_n(.)$ of $W_n$ is given by
\begin{equation}
\label{integral equation F(n+1)}
\begin{split}
F_{n+1}(x) &= \int_{0}^{\infty} dA(t)\int_{0}^{x+t} dW_{n}(u)\left[\bar{G}(u) \int_{t-u}^{x+t-u}dB(s) + G(u) + \bar{G}(u)\int_{0}^{t-u} V(x+t-u-s) dB(s) \right] \\
           & \qquad + \int_{0}^{\infty} dA(t)\int_{0}^{t} dW_{n}(u)\left[ -G(u)\bar{V}(x+t-u) \right].
\end{split}
\end{equation}

\n
Using the Helly-Bray Theorem (see for example \cite{LoeveI}) we obtain the asymptotic integral equation satisfied by the distribution function $F$
\begin{equation}
\label{integral equation F}
\begin{split}
F(x) &= \int_{0}^{\infty} dA(t)\int_{0}^{x+t} d(u)\left[\bar{G}(u) \int_{t-u}^{x+t-u}dB(s) + G(u) + \bar{G}(u)\int_{0}^{t-u} V(x+t-u-s) dB(s) \right] \\
           & \qquad + \int_{0}^{\infty} dA(t)\int_{0}^{t} dW(u)\left[ -G(u)\bar{V}(x+t-u) \right].
\end{split}
\end{equation}

\section{The case of arrivals according to a Poisson Point Process}
\label{section 4}

\subsection{Integral equation of \texorpdfstring{$M/GI/V_S/1 + GI$}{Lg}}

\begin{proposition}
\label{stationary integral equation}
The stationary pdf $f$ of the waiting time in the $M/GI/V_s + G$ system satisfies the integral equation

\begin{equation}
\label{stationary integral equaiton}
f(x) = \lambda P(0) \bar{B}(x) + \lambda_{2} \bar{V}(x) + \lambda \int_{0}^{x} \bar{G}(u) \bar{B}(x-u) f(u) du,\quad x > 0,
\end{equation}

\n
with the normalizing equation

\begin{equation}
P(0) + \int_{0}^{\infty} f(x)dx = 1,
\end{equation}

\n
where $P(0)$ is the steady state-state probability of finding the system empty and the unknown constant $\lambda_{2}$ is given by

\begin{equation}
\label{constant Cv}
\lambda_{2} := C_{1} + C_{2} + C_{3},
\end{equation}

\n
with

\begin{eqnarray}
C_{1} &:=& P(0)\int_{z=x}^{\infty}f_{B}(z-x)\lambda e^{-\lambda(z-x)}dz, \label{constant C1}\\
C_{2} &:=& \int_{z=x}^{\infty}\int_{u=0}^{z-x} \bar{G}(u)f_{B}(z-x-u)f(u)\lambda e^{-\lambda(z-x)}du dz,\label{constant C2}\\
C_{3} &:=& \int_{z=x}^{\infty} f(z-x) \lambda e^{-\lambda(z-x)}dz.\label{constant C3}
\end{eqnarray}

\end{proposition}

\begin{remark}
This result is consistent with the crossing level method (see \cite{Brill}). For instance,
\cite{Katayama2011} and \cite{Sarhangian2013} applied directly this method to derive integral equations.
\end{remark}

\subsection{Case of exponentially distributed patience time}

\begin{lemma}
The density function of the waiting time in the $M/GI/G/1 + M$ system satisfies the integral equation

\begin{equation}
\label{integral equation with exponential patience}
f(x) = \lambda P(0) \bar{B}(x) + \lambda_{2} \bar{V}(x) + \lambda \int_{0}^{x} \bar{B}(x-u) e^{-\gamma u}f(u) du,\quad x > 0,
\end{equation}

\n
with the normalizing equation

\begin{equation}
\label{normalizing equation}
F(0) + \int_{0}^{\infty} f(x)dx = 1.
\end{equation}

\end{lemma}

\begin{proposition}
In the $M/GI/G/1 + M$ system, the LST of the virtual waiting time, $f^{*}$ is given by

\begin{equation}
\label{LST of f}
f^{*}(\theta) = \sum_{j=0}^{\infty}\left( P(0) + \dfrac{\lambda_2 V^{*}(\theta + j\omega_2)}{\lambda B^{*}(\theta+j\omega_2)} \right)\prod_{m=0}^{j}\lambda B^{*}(\theta + m\omega_2),
\end{equation}

\n
where

\begin{equation}
\label{lambda_2}
\lambda_2 =   \dfrac{\lambda F(0)}{q_{0}}, \qquad \textrm{with}\;\; q_{0} = V^{*}(\lambda),
\end{equation}

\n
and

\begin{equation}
\label{F(0)}
F(0) = \left( 1+  \sum_{j=0}^{\infty}\left( 1 + \dfrac{ V^{*}(\theta + j\omega_2)}{V^{*}(\lambda) B^{*}(\theta+j\omega_2)} \right)\prod_{m=0}^{j}\lambda B^{*}(\theta + m\omega_2) \right)^{-1}.
\end{equation}

\end{proposition}

\begin{proof}
Equation~\eqref{LST of f} is obtained using the solution in \cite{Jagerman2000}. This equation helps to find $F(0)$, the probability of having an idle system. Indeed, letting $\theta \rightarrow 0$ in both sides of Equation~\eqref{LST of f} and using the normalizing equation~\eqref{normalizing equation}, we get $f^{*}(0) = 1 - F(0)$ which yields \eqref{F(0)}.
\end{proof}

\section{Tail behaviour of the waiting time distribution}
\label{section 5}

In this section, we investigate the tail behaviour of the steady state waiting time distribution. We use only information from the given distributions of the service and vacation times. 
This information is relevant when, for example, the distribution $F$ cannot be computed exactly.
We consider the class of heavy tail distributions which plays a major role in the analysis of many stochastic systems and especially in communication networks. 
The defining property of heavy-tailed distributions is that their tails decrease more slowly than an exponential tail. A subset of particular interest in queueing theory is that of long-tailed distributions. Intuitively, this means that if $X > x$, for some large $x$, then it is likely that X exceeds any larger value as well.
We focus on the impact of a long-tail service time distribution upon the distribution of the steady-state waiting time in the $M/G/1+M$ queue with single vacations under FIFO discipline. We do this by exploiting the functional equation for the steady state p.d.f of the waiting time derived in the previous section.
\vspace{0.2cm}

We start by introducing some notation. For any two real functions $f(.)$ and $g(.)$ we use the conventional notation $f(x) \underset{\infty}{\sim}g(x)$ to denote $\underset{x \rightarrow \infty} {\lim} f(x)/g(x) = 1$. For any positive random variable $X$ with distribution function $F$ and having finite expectation $\mathbb{E}(X) < \infty$, we define the integrated tail distribution $F^{r}$ by
\begin{equation*}
F^{r}(x) = \dfrac{1}{\mathbb{E}(X)}\int_{0}^{x}\bar{F}(y) dy, \quad x \geq 0,
\end{equation*} 
The mean of the service and vacation time are denoted by $\beta$ and $\vartheta$, respectively, and are such that
\begin{equation*}
\int x \, dB(x) = \beta < \infty \quad \mbox{ and } \quad \int x \, dV(x) = \vartheta < \infty.
\end{equation*}

\n We write $\mathcal{L}$ for the class of long tail distributions (see Definition~\ref{chap1-proba:definition long tailed distribution}). The function $\bar{F}$ is a tail function, if and only if (see \cite{Smith1972})
\begin{equation*}
\bar{F}(x) \sim \exp\left\lbrace - \int_{0}^{x} \lambda(u) du \right\rbrace,
\end{equation*}

\n 
where $\lambda(u) \downarrow 0$. Define also, $\Lambda(x) := \int_{0}^{x} \lambda(u) du$, and we consider all the distribution functions belonging to $\mathcal{L}$ and satisfying
\begin{equation}
\label{regularity condition smith72}
\varlimsup_{x \to \infty} \Lambda(2x)/\Lambda(x) < 2.
\end{equation}

\begin{remark}
All distribution functions belonging to $\mathcal{L}$ satisfying condition~\eqref{regularity condition smith72} are called by Smith in \cite{Smith1972} subexponential distributions. The standard definition for such distributions is the following: a distribution function $F$ on $[0,\infty)$ belongs to the subexponential class denoted by $\mathcal{S}$ if and only if $1-F^{(2)}(x) \sim 1-F(x)$, as $x \rightarrow \infty$, where $F^{(2)}$ is the convolution of $F$ with itself. The class of subexponential distribution functions is an important subclass of heavy-tailed distributions and was introduce by Chistyakov \cite{Chistyakov1964}.
A more complete account of subexponential distribution functions can be found in \cite{Foss} or in \cite{GoldieKluppelberg, Pitman1980, Kluppelberg1988, Teugels1975, Murphree1989}.

\end{remark}


We have derived in the previous subsection the integral equation of the steady-state p.d.f. of the waiting time for the $M/G/1+M$ with single vacations queueing model. This p.d.f is given by
\begin{equation}
\label{integral equation f}
f(z) = F(0)\lambda\bar{B}(z) + \lambda_{2}\bar{V}(z) + \lambda \int_{0}^{z} \bar{B}(z-u) e^{-\gamma u} f(u) du, \quad z >0,
\end{equation}
where the constants $\lambda_{1}$, $\lambda_{2}$ and $\gamma$ are supposed to be known. Starting from this equation it will be shown that if $B^r$ is a long-tailed distribution such that the condition~\eqref{regularity condition smith72} is fulfilled then $B^r$ and $F$ are asymptotically equivalent and $F$ is a long-tailed distribution, and conversely. Let $f^*$ be the Laplace-Stieltjes Transform (LST) of the stationary waiting time distribution $F$
\begin{equation*}
f^{*}(s) = \int_{0^-}^{\infty} e^{-s x} dF(x), \;\; \Re(s) \geq 0.
\end{equation*}

\n Define $F_{\gamma}$ the distribution function associated with $F$, by
\begin{equation}
dF_{\gamma} (x) := e^{-\gamma x} dF(x) / f^*(\gamma), \quad x \geq 0,
\end{equation}

\n where $\gamma$ is the parameter of the patience time distribution. Equation~\eqref{integral equation f} becomes
\begin{equation}
\label{integral equation of f wrt f_gamma}
f(z) = F(0)\lambda\bar{B}(z) + \lambda_{2}\bar{V}(z) + \lambda f^{*}(\gamma) \int_{0}^{z}\bar{B}(z-u) \dfrac{f_{\gamma}(u)}{f^{*}(\gamma)} \, du, \quad z > 0.
\end{equation}

\n Integrating \eqref{integral equation of f wrt f_gamma} over $]0,x]$ with respect to the Lebesgue measure leads to
\begin{equation}
\label{integral equation of the df F wrt df_gamma}
F(x) =  F(0) + \rho_{2}V^{r}(x) + f^{*}(\gamma) \rho_{1}\int_{0^-}^{x} B^{r}(x-u)dF_{\gamma}(u), \quad x \geq 0,
\end{equation}

\n
where $\rho_{1} := \lambda \beta$, $\rho_{2} := \lambda_{2} \vartheta$ are such that $\rho_{1} + \rho_2 < 1$. Note that $F(0)$ and $f^*(\gamma)$ are still unknown. However, letting $t \uparrow \infty$ in \eqref{integral equation of the df F wrt df_gamma}, we have $1 = F(0) + \rho_{2}  \rho_{1} f^{*}(\gamma)$. Moreover, using the characterisation of $\lambda_2$ given by Equation~\eqref{lambda_2}, we find both $F(0)$ and $f^*(\gamma)$. 
Note that the value of the LST at point $\gamma$ depends on $F(0)$ since $f^*(\gamma) = F(0) + \int_{0^+}^{\infty} \, e^{-sx}dF(x)$. Using the condition $1 = F(0) + \rho_{2}  \rho_{1} f^{*}(\gamma)$ and performing an integration by part, we can rewrite Equation~\eqref{integral equation of the df F wrt df_gamma}, as follows in terms of $\bar{F}$, as follows
\begin{equation}
\label{integral equation of F-bar wrt dB^r}
\bar{F}(x) = \rho_{1}f^{*}(\gamma)\bar{B}^{r}(x) + \rho_{2}\bar{V}^{r}(x) +  \rho_{1}f^{*}(\gamma) \int_{0}^{x}\bar{F}_{\gamma}(x-u)dB^{r}(u), \quad x \geq 0.
\end{equation}

\n An alternative of \eqref{integral equation of F-bar wrt dB^r} is given by
\begin{equation}
\label{integral equation of F-bar wrt dF_gamma}
\bar{F}(x) = \rho_{1}f^{*}(\gamma)\bar{F}_{\gamma}(x) + \rho_{2}\bar{V}^{r}(x) +  \rho_{1}f^{*}(\gamma) \int_{0^-}^{x}\bar{B}^{r}(x-u)\,dF_{\gamma}(x-u), \quad x \geq 0.
\end{equation}

\begin{remark}
Equation~\eqref{integral equation of F-bar wrt dB^r} is of convolution type and looks like a renewal equation. There is however two important differences. Since $f^*(\gamma)\rho_1 < 1$, $\int_{0^-}^{x}\bar{F}_{\gamma}(x-u)d\left(f^*(\gamma)\rho_1 B^{r}(u)\right)$ is an integral with respect to a defective distribution. Furthermore, the integrand in the integral does not coincide with $\bar{F}$ which is the left hand side of \eqref{integral equation of F-bar wrt dB^r}. As a consequence, \eqref{integral equation of F-bar wrt dB^r} is not a renewal equation (or rather, a defective renewal equation) and does not open the door to renewal theory. The steady-state waiting-time tail probability for standard queueing models (without both impatience and vacations) has been widely investigated under various assumption on the service time distribution.
\end{remark}

We state the asymptotic equivalence under the following assumptions:

\begin{description}
 \item [H.1] The distribution function $V$ is lighter than $B$ in the sense that $\underset{x \to \infty}{\lim} \dfrac{\bar{V}^{r}(x)}{\bar{B}^{r}(x)} = 0 $. 
\item [H.2] The distribution function $V$ is lighter than $F$ in the sense that $\underset{x \to \infty}{\lim} \dfrac{\bar{V}^{r}(x)}{\bar{F}(x)} = 0$.
\item [H.3] There exist $M > 0$ such that $\dfrac{\bar{F}(x)}{\bar{B}(x)} \leq M$ for all $x \in [0, +\infty[$.
\end{description}

\begin{theorem}
\label{chap2-proba:theorem asymptotic equivalence}
Assume assumptions \textbf{H.1}-\textbf{H.3} hold. If $B^r \in \mathcal{L}$ and fulfills condition~\eqref{regularity condition smith72}, then 
\begin{equation}
\label{chap2-proba: theorem tail equivalence between F-bar and B-bar}
 \bar{F}(x) \sim \rho_{1}f^{*}(\gamma) \bar{B}^{r}(x), \quad x \to \infty.
\end{equation} 
and $F \in \mathcal{L}$. Conversely, If $F \in \mathcal{L}$ such that condition~\eqref{regularity condition smith72} is fulfilled, then \eqref{chap2-proba: theorem tail equivalence between F-bar and B-bar} holds and $B^r \in \mathcal{L}$.
\end{theorem}

\n
To prove this theorem we need to

\begin{lemma}
\label{chap2-proba:lemma1 tail asymptotic:limitinf}
Under assumptions \textbf{H.1}-\textbf{H.3}, we have
\begin{equation*}
\liminf_{x \to \infty} \dfrac{\bar{F}(x)}{\bar{B}^{r}(x)} \geq \rho_{1}f^{*}(\gamma).
\end{equation*}
\end{lemma}

\begin{lemma}
\label{chap2-proba:lemma2 tail asymptotic:limitsup}
Under assumptions \textbf{H.1}-\textbf{H.3}, if $B^r \in \mathcal{L}$ and satisfies \eqref{regularity condition smith72}, then
\begin{equation*}
\limsup_{x \to \infty} \dfrac{\bar{F}(x)}{\bar{B}^{r}(x)} \leq \rho_{1} f^{*}(\gamma).
\end{equation*}
\end{lemma}

\begin{lemma}
\label{chap2-proba:lemma3 tail asymptotic:limitinf with F longtail}
Under assumptions \textbf{H.1}-\textbf{H.3}, if $F \in \mathcal{L}$ such that condition~\eqref{regularity condition smith72} holds, then
\begin{equation*}
\liminf_{x \to \infty} \dfrac{\bar{B}^{r}(x)}{\bar{F}(x)} \geq \dfrac{1}{\rho_{1} f^{*}(\gamma)}.
\end{equation*}
\end{lemma}

\begin{appendices}

\renewcommand\thesection{\Alph{section}}

\setcounter{equation}{0}
\section{Proofs}

\renewcommand{\thesection}{A.\arabic{section}}

\section{Proof of Lemma~\ref{chap2-proba:lemmma stability of Yn}}

The proof follows Theorem 4.3.1. of \cite{Moyal2005} and \cite{BaccelliBremaud} (P.74 to 78).
The first part of the proof is devoted to the existence of the stationary solution to Equation~\eqref{chap2-proba:stationary equation Y}. The mapping $x \mapsto \psi(x)$ is non-negative, continuous and non-decreasing. Hence, \eqref{chap2-proba:stationary equation Y} can be solved using the Loyes's construction. For more details about Loynes's technique in the $G/G/1/\infty$ setting, the reader may refer to the Fundamental Result of Stability in \cite{BaccelliBremaud}(p. 71 and paragraph 2 p.74 for the proof). 

\paragraph*{Existence of a stationary solution.}
Let us define the sequence of workload $\left( Y^{0}_{n}, n \geq 0 \right)$ when the initial condition is $Y_{0} = 0$. Define also, the Loynes's sequence $\left( M_n, n \in \mathbb{N} \right)$, associated to $\left( Y^0_n, n \in \mathbb{N} \right)$, by
\begin{equation}
\label{chap2-proba:Definition Loynes's sequence for Y_n}
M_0 = 0, \qquad M_n = Y_n^0 \circ \theta^{-n},
\end{equation}

\n
The shift operator $\theta^{-n}$ means that the time is shifted $n$ times to the left, \textit{i.e.}, the new time origin is the arrival time of customer $C_{-n}$, so that $M_n$ corresponds to the workload finds by the customer $C_0$ when customer $C_{-n}$ finds an empty queue. We can prove, using \eqref{chap2-proba:Definition Loynes's sequence for Y_n}, that $M_n$ satisfies the following recursive relation, for $n \geq 0$,
\begin{equation}
\label{chap2-proba: recursive relation Mn}
M_n = \left[ \max_{1 \leq j \leq n}\left( \sigma_{-j} + D_{-j} + \sum_{i=1}^{j }V_{-i}  - \sum_{i=1}^j \tau_{-i} \right)\right]^+.
\end{equation}

\n Moreover $\left( M_n, n \geq 0 \right)$ satisfies, for all $n \geq 0$
\begin{equation}
\label{chap2-proba:eq-recurence Mn circ theta}
M_{n+1} \circ \theta = \left[ (M_n+V) \vee (\sigma +D +V) - \tau \right]^+.
\end{equation}

\n We prove \eqref{chap2-proba:eq-recurence Mn circ theta} by induction on $n$. For $n=0$, we have
\begin{align*}
M_{1} \circ \theta &= \left[ \max(M_0 + V, \sigma + D +V) - \tau \right]^+ \\
                   &= \left[ \max(V, \sigma + D +V) - \tau \right]^+ \\
                   &= \left[  \sigma + D +V- \tau \right]^+,
\end{align*}
so that
\begin{align*}
M_{1} &= \left( \sigma + D + V - \tau \right)^{+} \circ \theta^{-1}\\
      &= \left( \sigma_{-1} + D_{-1} +V_{-1} - \tau_{-1} \right)^+.
\end{align*}

\n Assume that \eqref{chap2-proba:eq-recurence Mn circ theta} holds for some $n$ and prove that
\begin{equation*}
M_{n+1} \circ \theta = \left[ (M_n+V) \vee (\sigma +D+V) - \tau \right]^+.
\end{equation*}
We have from \eqref{chap2-proba: recursive relation Mn}
\begin{align*}
M_{n+1} &= \left[ \max_{1 \leq j \leq n+1}\left( \sigma_{-j} + D_{-j} + \sum_{i=1}^{j }V_{-i}  - \sum_{i=1}^j \tau_{-i} \right)\right]^+\\
        &= \left[ \max\left\lbrace \sigma_{-1} + D_{-1} +V_{-1}  -\tau_{-1}, \max_{2 \leq j\leq n+1} \left( \sigma_{-j} + D_{-j} + \sum_{i=1}^{j }V_{-i}  - \sum_{i=1}^j \tau_{-i}\right) \right\rbrace \right]^+\\
        &= \left[ \max \left\lbrace \sigma_{-1} + D_{-1} +V_{-1}  -\tau_{-1}, \max_{1 \leq j\leq n} \left( \sigma_{-j-1} + D_{-j-1} + \sum_{i=0}^{j}V_{-i-1}  - \sum_{i=0}^j \tau_{-i-1}\right) \right\rbrace \right]^+\\
        &= \left[ \max \left\lbrace \sigma_{-1} + D_{-1} +V_{-1}  -\tau_{-1}, \max_{1 \leq j\leq n} \left( \sigma_{-j-1} + D_{-j-1} + \sum_{i=1}^{j}V_{-i-1}  - \sum_{i=1}^j \tau_{-i-1} + V_{-1} - \tau_{-1} \right) \right\rbrace \right]^+\\
        &= \left[ \max \left\lbrace \sigma_{-1} + D_{-1} +V_{-1}  -\tau_{-1}, \max_{1 \leq j\leq n} \left( \sigma_{-j-1} + D_{-j-1} + \sum_{i=1}^{j}V_{-i-1}  - \sum_{i=1}^j \tau_{-i-1} \right)+ V_{-1} - \tau_{-1}  \right\rbrace \right]^+\\
        &= \left[ \max \left\lbrace \max_{1 \leq j\leq n} \left( \sigma_{-j} + D_{-j} + \sum_{i=1}^{j}V_{-i}  - \sum_{i=1}^j \tau_{-i} \right)+ V - \tau, (\sigma + D +V)  -\tau \right\rbrace \right]^+ \circ \theta \\
        &=\left[ \max \left(M_n + V,\sigma + D +V  \right) - \tau \right]^+ \circ \theta,
\end{align*}
and thus,
\begin{equation*}
M_{n+1} \circ \theta = \left[ \max \left(M_n + V,\sigma + D +V  \right) - \tau \right]^+ .
\end{equation*}

Next, we will prove that the sequence  $\{ M_n, n \in \mathbb{N} \}$ is non-decreasing. For this, we use the identity $\left[ \max(a,b) \right]^+ = \max(a^+,b)$ and Equation~\eqref{chap2-proba: recursive relation Mn}. For all $n \in \mathbb{N}$, we have
\begin{align*}
M_{n+1} &= \left[ \max_{1 \leq j \leq n+1}\left( \sum_{-j} + D_{-j}+ \sum_{i=1}^{j} V_{-i}- \sum_{i=1}^j \tau_{-i}  \right)\right]^+\\
       &= \left[ \max \left\lbrace \max_{1 \leq j \leq n}\left( \sigma_{-j} + D_{-j} +  \sum_{i=1}^{j} V_{-i} - \sum_{i=1}^j \tau_{-i}  \right), \sigma_{-n-1} + D_{-n-1} + \sum_{i=1}^{n+1} V_{-i} - \sum_{i=1}^{n+1} \tau_{-i} \right\rbrace \right]^+\\
        &=  \max \left\lbrace \left[\max_{1 \leq j \leq n}\left( \sigma_{-j} + D_{-j} + \sum_{i=1}^{j} V_{-i} - \sum_{i=1}^j \tau_{-i} \right)\right]^+, \sigma_{-n-1} + D_{-n-1} +\sum_{i=1}^{n+1} V_{-i} - \sum_{i=1}^{n+1} \tau_{-i} \right\rbrace \\   
        &=  \max \left\lbrace M_n, \sigma_{-n-1} + D_{-n-1}+ \sum_{i=1}^{j} V_{-i} - \sum_{i=1}^{n+1} \tau_{-i} \right\rbrace  \geq M_n.
\end{align*}

\n
The sequence being non-decreasing, there exists a non-negative random variable (possibly infinite) denoted by $M_{\infty}$, such that
\begin{equation}
\label{eq-def-M_infini}
M_{\infty} = \lim_{n \rightarrow \infty}\uparrow M_n = \left[ \sup_{j > 0}\left( \sigma_{-j} + D_{-j} +\sum_{i=1}^{j} V_{-i} - \sum_{i=1}^j \tau_{-i} \right)\right]^+.
\end{equation}

\n
Taking the limit as $n$ goes to $\infty$ in \eqref{chap2-proba:eq-recurence Mn circ theta}, and using the continuity of $\psi$, the limiting value $M_{\infty}$ satisfies
\begin{equation}
\label{limite M-infini}
M_{\infty} \circ \theta = \left[ \max \left( M_{\infty}+V ,\sigma + D + V \right) - \tau \right]^+  .
\end{equation}
Hence, $M_{\infty}$ seems to be a reasonable candidate for the stationary random variable $Y$. It remains to show that $M_{\infty} < \infty$ $\mathbb{P}$-a.s. and $M_{\infty}$ is  $\mathbb{P}$-a.s. unique.

\paragraph*{The random variable $\boldsymbol{M_{\infty}}$ is $\boldsymbol{\mathbb{P}}$-a.s. finite.} 

To prove this property, we show that the event $\{ M_{\infty} = \infty \}$ is $\theta$-invariant, \textit{i.e.},  $\theta^{-1}\{ M_{\infty} = \infty \} =  \{ M_{\infty} = \infty \}$. Given the ergodicity of $(\mathbb{P}, \theta)$, we will have $\mathbb{P}(M_{\infty} = \infty ) = 0$ or $1$. It will remain only to prove that $\mathbb{P}(M_{\infty} = \infty ) = 0$ using the Birkhoff's theorem.

\n\\
The event $\{ M_{\infty} = \infty \}$ is $\theta$-invariant. Indeed,
\begin{align*}
\theta^{-1}\{ M_{\infty} = \infty \} &= \{\omega \mid \theta\omega \in \{ M_{\infty} = \infty \} \}\\
              &= \{\omega \mid  M_{\infty}(\theta\omega) = \infty  \}\\
              &= \{\omega \mid  M_{\infty}\circ \theta(\omega) = \infty  \}\\
              &= \{ M_{\infty}\circ \theta = \infty  \}\\
              &=  \left\lbrace \left[ \max( M_{\infty}+V, \sigma + D + V) - \tau  \right]^+ = \infty \right\rbrace\\
              &= \{ M_{\infty} = \infty \},
\end{align*}

since the random variables $\tau$, $\sigma$, $V$ and $D$ are $\mathbb{P}$-integrable. 
Therefore, in view of the ergodic assumption of $(\mathbb{P}, \theta)$ we have $\mathbb{P}( M_{\infty} = \infty) = 0$ or $1$. 
Noting that
\begin{equation}
\label{telescopic sum for sigma}
\sum_{i=1}^{j} \left( \sigma_{-i} - \sigma_{-i+1} \right)+ \sigma_{0} = \sigma_{-j},
\end{equation}
and
\begin{equation}
\label{telescopic sum for D}
\sum_{i=1}^{j} \left( D_{-i} - D_{-i+1} \right) \sigma_{0} = \sigma_{-j}.
\end{equation}

\n We have, by substituting the values of $\sigma_{-j}$ and $D_{-j}$ by \eqref{telescopic sum for sigma} and \eqref{telescopic sum for D}, respectively, in Equation~\eqref{eq-def-M_infini},
\begin{equation}
\label{eq-M_infini modifiee}
M_{\infty} = \left[ \sup_{j > 0}\left( \sum_{i=1}^j (\sigma_{-i}- \sigma_{-i+1}+ D_{-i} - D_{-i+1} + V_{-i} - \tau_{-i})+ \sigma_{0} + D_{0} \right)\right]^+.
\end{equation}

\n To prove that $M_{\infty} < \infty$ $\mathbb{P}$-a.s., suppose that $M_{\infty} = \infty$ $\mathbb{P}$-a.s., thus from \eqref{eq-M_infini modifiee}, we have
\begin{equation*}
\sup_{j > 0}\left( \sum_{i=1}^j (\sigma_{-i}- \sigma_{-i+1}+ D_{-i} - D_{-i+1} + V_{-i} - \tau_{-i})+ \sigma_{0} + D_{0} \right)= \infty, \qquad \mathbb{P}-a.s.,
\end{equation*}

\n hence
\begin{equation}
\label{sum telescopic = + infty}
\sup_{j > 0}\left( \sum_{i=1}^j (\sigma_{-i}- \sigma_{-i+1}+ D_{-i} - D_{-i+1} + V_{-i} - \tau_{-i}) \right) = \infty,
\end{equation}
since $\sigma$ and $D$ are $\mathbb{P}$-a.s. finite. Moreover, the Birkhoff's theorem leads to
\begin{equation*}
\lim_{j \rightarrow \infty}  \frac{1}{j} \sum_{i=1}^j (\sigma_{-i}- \sigma_{-i+1}+ D_{-i} - D_{-i+1} + V_{-i} - \tau_{-i}) = \mathbb{E}^0(\sigma - \sigma + D - D + V - \tau ) < 0 \qquad \mathbb{P}-a.s.,
\end{equation*}
since we have assumed that $\mathbb{E}\left( V-\tau\right)< 0$. Therefore
\begin{equation*}
\lim_{j \rightarrow \infty}\sum_{i=1}^j (\sigma_{-i}- \sigma_{-i+1}+ D_{-i} - D_{-i+1} + V_{-i} - \tau_{-i}) = -\infty \qquad \mathbb{P}-a.s,
\end{equation*}
and thus
\begin{equation*}
\sup_{j >0}\sum_{i=1}^j (\sigma_{-i}- \sigma_{-i+1}+ D_{-i} - D_{-i+1} + V_{-i} - \tau_{-i}) = -\infty \qquad \mathbb{P}-a.s,
\end{equation*}
which contradicts \eqref{sum telescopic = + infty}. Thus $M_{\infty} < \infty$ $\mathbb{P}$-a.s..

\paragraph*{The random variable $\boldsymbol{M_{\infty}}$ is unique.}

We have to prove firstly that $\mathbb{P}\left( Y \leq (\sigma + D + V )\right) > 0$, where $Y$ is given by Equation~\eqref{eq-def-M_infini}. We consider two cases. \\

$\bullet$ If $\mathbb{P}(Y = 0) > 0$, we then have $0 < \mathbb{P}(Y =0 ) \leq \mathbb{P}( Y \leq \sigma + D + V)$ since the random variables $\sigma$, $D$ and $V$ are positive. \\

$\bullet$ If $\mathbb{P}(Y = 0) = 0$, i.e. $Y > 0$ $\mathbb{P}$-a.s.. In this case, we assume that $\mathbb{P} \left( Y \leq \sigma + D + V \right) = 0$, hence $Y  > \sigma + D + V $, $\mathbb{P}$-a.s.. Using \eqref{limite M-infini} we have
\begin{equation*}
  Y  \circ \theta = \left[ (Y+V) \vee \left(\sigma + D + V \right) - \tau \right] = \left[ Y + V - \tau \right]^{+} = Y + V - \tau.
\end{equation*}

\n
The last equality holds since the sequence $\{ M_n\}$ is non-decreasing, as a consequence by using the assumption we have $ 0< Y \leq  Y\circ \theta $,  $\mathbb{P}$-a.s., hence $Y \circ \theta  -  Y = V - \tau$. Therefore, $\mathbb{E}(Y \circ \theta - Y) = \mathbb{E}(V- \tau) < 0$, which contradicts the Ergodic Lemma: $\mathbb{E}(Y \circ \theta - Y) = 0$, thus we have $\mathbb{P}\left( Y \leq \sigma +D + V )\right) > 0$.  \vspace{0.3cm}

We then have to show that $M_{\infty}$ is the minimal non-negative solution of \eqref{limite M-infini}, \textit{i.e.}, $M_{\infty} \leq U$, $\mathbb{P}$-a.s. for some non-negative solution of \eqref{limite M-infini}. We shall prove this property by induction on $n$. Let $U$ be a finite non-negative solution of \eqref{limite M-infini}, then $U$ satisfies 
\begin{equation*}
U \geq 0, \qquad U \circ \theta = (U+V) \vee \left(\sigma + D + V \right) - \tau , \quad \mbox{and} \quad \mathbb{P} \left(U \leq \sigma + D + V \right) > 0.
\end{equation*}

\n
For $n=0$, this property is obvious since $U \geq 0 = M_0$. Suppose that, $\mathbb{P}$-a.s. $U \geq M_n$ holds for some $n > 0$ , we obtain using the inductive assumption
\begin{equation*}
 U \circ \theta = \left[(U+V) \vee \left(\sigma + D+ V \right) - \tau \right]^+ \geq \left[(M_n+V) \vee \left(\sigma + D + V \right) - \tau \right]^+ = M_{n+1} \circ \theta, \quad \mathbb{P}- \mbox{a.s.}
\end{equation*}

\n hence $U \geq M_{n+1}$ $\mathbb{P}$-a.s.. From monotone convergence theorem, we have 
\begin{equation}
\label{eq-minimility of M_inf}
U \geq M_{\infty}, \quad \mathbb{P}-\mbox{a.s.},
\end{equation}

\n thus $M_{\infty}$ is the minimal solution of \eqref{limite M-infini}.  Therefore, it suffices to show that $\mathbb{P}(U \leq M_{\infty}) = 1$ to obtain $\mathbb{P}(U = M_{\infty} ) = 1$. The event $\{ U \leq M_{\infty}  \}$ is $\theta$-contracting ($A$ is $\theta$-contracting if $A \subset \theta^{-1}A$). Indeed, on $\{ U \leq M_{\infty}  \}$, we have
\[ U \circ \theta = \left[ (U+V) \vee \sigma + D+V \right]^+  \leq  \left[ (M_{\infty}+V) \vee (\sigma + D+V\right]^+ = M_{\infty} \circ \theta,\qquad \mathbb{P}-\mbox{a.s.},\]
i.e. $\{U \leq M_{\infty}  \} \subset \theta^{-1}\{U \leq M_{\infty}  \} = \{ U \circ \theta \leq   M_{\infty} \circ \theta \}$. In view of the ergodic Lemma and Remark 2.3.1 p.77 in \cite{BaccelliBremaud}, the event $\{Z \leq M_{\infty}  \}$ is $\theta$-invariant, \textit{i.e.},
\[ \{U \leq M_{\infty}  \}  = \{ U \circ \theta \leq   M_{\infty} \circ \theta \}. \]
Since $(\mathbb{P}, \theta)$ is ergodic $\mathbb{P}(U \leq M_{\infty} )=0$ or $1$. Furthermore, on $\{ U \leq M_{\infty}  \}$, we have

\begin{align*}
\begin{split}
U \circ \theta &= \left[ (U+V) \vee (\sigma + D + V) - \tau  \right]^+\\
               &= \left[ U+V  - \tau  \right]^+ \mathbf{1}_{\{ U+V > \sigma + D + V \}} + \left[ \sigma + D + V - \tau  \right]^+ \mathbf{1}_{ \{ U+V \leq \sigma + D + V \}}\\
               &\leq \left[ M_{\infty}+V  - \tau  \right]^+ \mathbf{1}_{\{ U > \sigma + D + V \}} + \left[ \sigma + D + V - \tau  \right]^+  \mathbf{1}_{ \{ U+V \leq \sigma + D + V \}} \\
               &  \leq \left(M_{\infty} \circ \theta\right) \mathbf{1}_{\{ U > \sigma + D + V \}}  + \left(M_{\infty} \circ \theta \right) \mathbf{1}_{ \{ U \leq \sigma + D + V \}} = M_{\infty} \circ \theta
\end{split}               
\end{align*}
where the last inequality holds because of $M_{\infty} \leq (M_{\infty}+V) \vee (\sigma + D + V )$ and $\left(\sigma + D + V \right) \leq (M_{\infty} +V) \vee (\sigma + \tilde{D} + V ) $. From the last inequality, and the fact that the event $\{ U \leq M_{\infty} \}$ is $\theta$-invariant, we have
\begin{equation*}
\{ U \leq (\sigma + D + V ) \} \subset \{ U \circ \theta  \leq  M_{\infty} \circ \theta\} = \{ Z \leq  M_{\infty} \},
\end{equation*}

\n
which yields to 
\begin{equation}
0 <  \mathbb{P}(U \leq (\sigma + D + V )) \leq \mathbb{P}( Z \leq  M_{\infty}),
\end{equation}
since $\mathbb{P}\left(Z \leq \sigma + D + V \right) > 0$. Moreover, the ergodicity of $(\mathbb{P}, \theta)$ implies that $\mathbb{P}( U \leq  M_{\infty}) > 0$. Thus with \eqref{eq-minimility of M_inf}, we have
\begin{equation*}
 \mathbb{P}(Z = M_{\infty}) = 1. 
\end{equation*}

\section{Proof of Lemma~\ref{chap2-proba:lemma:comparaison Z_n < Y_n}}

We prove only that $W_{1} \leq Y_{1}$. The proof is the same for $Z_n \leq Y_n$. We have $X = W+ \sigma\mathbf{1}_{W \geq D} - \tau > 0$. \\

$\bullet$ If $W \leq D$, then $Z_1 = W+ \sigma - \tau$ and 
\begin{equation*}
0 < Z_1 = W+ \sigma - \tau \leq D + \sigma - \tau \leq D + V+ \sigma - \tau \leq \max(Y+V,\sigma+V+D)-\tau
= Y_1.
\end{equation*}

$\bullet$ If $W \geq D$ then $Z_1 = W - \tau > 0$.
\begin{equation*}
0< Z_1 = W- \tau \leq Y-\tau \leq Y+V-\tau \leq \max(Y+V,\sigma+V+D)-\tau = Y_1.
\end{equation*}

$\bullet$ If $X \leq  0$ and $W \leq D$ then $Z_1 = W + \sigma + V -\tau$.
\begin{equation*}
Z_1 = W + \sigma + V -\tau \leq D \sigma + V -\tau \leq \max(Y+V,\sigma+V+D)-\tau = Y_1.
\end{equation*}

\section{Proof of Theorem~\ref{chap2-proba:theorem:stability of Z}}

\textit{Existence.} According to Lemma~\ref{chap2-proba:lemma:comparaison Z_n < Y_n}, we have
\begin{equation*}
 Z_{n}^{W} \leq  Y_{n}^{Y}, \quad \mathbb{P}-a.s.,
\end{equation*}

\n Assuming that $W \leq Y $ $\mathbb{P}$-a.s., we have

\begin{equation*}
Z_{1}^{W}\leq  Y \circ \theta,
\end{equation*}

\n hence by induction 
\begin{equation}
\label{majoration Zn < Y o theta}
Z_{n}^{W} \leq  Y \circ \theta^{n}, \quad n \geq 0.
\end{equation}

\n Define the event 
\begin{equation*}
A_n = \{ Y \circ \theta^n = 0 \},\quad \mbox{for all} \; n \in \mathbb{N}.
\end{equation*}

\n
The sequence $\{ A_n, n \in \mathbb{N} \}$ is a stationary sequence. Moreover it is a sequence of renovating events of length $1$ for the SRS $\{ Y_n^Y, n \in \mathbb{N} \}$ since the workload at the arrival time of the customer $C_{n+1}$ only depends on $\xi = (\sigma_n, D_n , V_n ,\tau_n)$ on $A_{n}$, i.e. $ Y_{n+1}^{Y} = \psi(0,\xi) = \left[ (0+V_n) \vee (\sigma_{n}+D_n +V_n) - \tau_{n} \right]^{+}$ on $A_n$. Moreover, on the event $A_n$, we have $Z_{n}^{W} = 0$. Therefore the sequence $\{ A_n, n \in \mathbb{N} \}$ is also a sequence of renovating event of length $1$ for the SRS $\{ Z_n^{W}, n \in  \mathbb{N} \}$, because $Z_{n+1}^{W}$ only depends on the driving sequence $\xi_n$. Indeed, on $A_n$, since $Z^{W}_{n} = 0$, then $Z^{W}_{n+1} = \varphi(0,\xi)$. According to the definition of $Y$ given by \eqref{chap2-proba:stationary solution for Yn}, Equation~\eqref{stability condition} is equivalent to $\mathbb{P}(Y = 0) > 0$. But the last event is exactly $A_0$, thus $\mathbb{P}(A_{0}) > 0$. As a consequence, we may apply Theorem 1 in \cite{Borovkov78} (or $(4.3.15)$ p. 101 in \cite{BaccelliBremaud}) and conclude that there exists a stationary sequence $\{ U \circ \theta^{n} \}$, solution of \eqref{chap2-proba:stationary equation Z}, and such that, for any initial condition $W$, the sequence $\{ W_n^{W} \}$ converges with strong coupling to $\{ U \circ \theta^{n} \}$.
\vspace{0.3cm}

\textit{Uniqueness.} Let $X$ and $X'$ be two finite and positive solutions of \eqref{chap2-proba:stationary equation Z}. We prove that $Z^{X}_{n}$ and $Z^{X'}_{n}$ admit the same sequence of renovating events. First, let us prove that if $X$ is a finite stationary solution of \eqref{chap2-proba:stationary equation Z} then $X \leq Y$ a.s.. To this end, we prove that the event $\{ X\leq Y \}$ is $\theta$-contracting, \textit{i.e.},  $\{ X \leq Y \} \subset  \theta^{-1}\{ X \leq Y \} = \{  X \circ \theta \leq Y \circ \theta \}$ and has a positive probability. 
On the event $\{ X \leq Y \}$, we have
\begin{equation*}
 X \circ \theta = \varphi(X, \xi) \leq \left[ (Y+V) \vee (\sigma + D + V) - \tau \right]^{+} = Y\circ \theta,
\end{equation*}

\n
which completes the proof of the $\theta$-contracting. In view of the ergodic Lemma (Lemma 2.3.1 p. 77 in \cite{BaccelliBremaud}) together with Remark 2.3.1 p.77, the event $\{ X\leq Y \}$ is $\theta$-invariant, \textit{i.e.}, 
\begin{equation*}
\{ X\leq Y \} = \{ X \circ \theta \leq Y \circ \theta \}.
\end{equation*}

\n
Since $(P^{0}, \theta)$ is ergodic, $\mathbb{P}( X\leq Y ) = 0$ or $1$. For proving that $\mathbb{P}( X\leq Y ) = 1$, it is enough to show that $\mathbb{P}( X\leq Y ) >0$. First, we show that $\mathbb{P}(X < D) >0$. Assuming that $\mathbb{P}(X < D) = 0$, \textit{ i.e.}, $X \geq D$ $\mathbb{P}$-a.s., we have
\begin{equation*}
X\circ \theta - X =  -\tau \mbox{ if } X > 0.     
\end{equation*}

\n
Since $\mathbb{E}^{0}(-\tau) < 0$, then $\mathbb{E}\left[ X \circ \theta - \theta  \right] \neq 0$ which contradicts the Ergodic Lemma. On the event $\{ X \leq D \}$, we have
\begin{align*}
 X \circ \theta  &\leq \varphi(D,\xi) \leq \left[(Y+V) \vee (\sigma + D +V)  \right]^+ = Y \circ \theta.
\end{align*} 

\n
Hence 
\begin{equation*}
0 < \mathbb{P}(X < D ) \leq \mathbb{P}(X \circ \theta \leq Y \circ \theta) = \mathbb{P}(X \leq Y).
\end{equation*}

\n We have also
\begin{equation}
Z_n^{X} = X \circ \theta \leq Y\circ \theta^{n}, \quad n \geq 0.
\end{equation}

\n
Therefore, for any finite solution $X$ of \eqref{chap2-proba:stationary equation Z}, the sequence $\{ Z{n}^{X} = X \circ \theta \}$ admits $A_n = \{Y \circ \theta^{n} \}$ as a stationary sequence of renovating events of length $1$. In view of (4.3.6) p.99 in \cite{BaccelliBremaud},
\begin{equation*}
\lim_{n \rightarrow \infty} Z_{n}^{X} \circ \theta^{-n} = X.
\end{equation*}

\n
Hence, if $X'$ is another finite solution of \eqref{chap2-proba:stationary equation Z}, we obtain that the sequence $\{ Z_{n}^{X'} = X' \circ \theta \}$ admits the same stationary sequence of renovating events as $\{ Z_{n}^{X}\}$. Thus,
\begin{equation*}
X = \lim_{n \rightarrow \infty} Z_{n}^{X}\circ \theta^{-n} = \lim_{n \rightarrow \infty} Z_{n}^{X'}\circ \theta^{-n} = X'.
\end{equation*}

\section{Proof of Proposition~\ref{stationary integral equation}}

\begin{proof}

Equation~\eqref{I(x)} yields

\begin{align*}
\begin{split}
I(x) & = \int_{0}^{\infty}dA(t)\int_{0}^{x+t}dF_n(u)\bar{G}(u) \left[B(x+t-u) - B(t-u) \right]\\
     & = \int_{0}^{\infty}dA(t)\int_{0}^{x+t}dF_n(u)\bar{G}(u)B(x+t-u) \\
     & \qquad -  \int_{0}^{\infty}dA(t)\int_{0}^{x+t}dF_n(u)\bar{G}(u)B(t-u)\\
     &=: I_1(x) + I_2(x).
\end{split}
\end{align*}

\n
The classical result of queueing with impatience (see for example \cite{Baccelli84}) gives

\begin{align*}
I_1 &= \int_{z=x}^{\infty} \int_{u=0^-}^{z} \bar{G}(u)B(z-u)f_n(u)\lambda e^{-\lambda(z-x)} du dz\\
    &=P_n(0) \int_{z=x}^{\infty} B(z)\lambda e^{-\lambda(z-u)}dz + \int_{z=x}^{\infty} \int_{u=0^+}^{z} \bar{G}(u)B(z-u)f_n(u)\lambda e^{-\lambda(z-x)} du dz.\\
\end{align*}

\n
Differentiating with respect to $x$ yields

\begin{align}
\begin{split}
I_1'(x) &= \lambda P_n(0) \int_{z=x}^{\infty} B(z)\lambda e^{-\lambda(z-x)}dz + \lambda \int_{z=x}^{\infty} \int_{u=0^+}^{z} \bar{G}(u)B(z-u)f_n(u)\lambda e^{-\lambda(z-x)} du dz \\
   & \qquad - \lambda P_n(0) B(x) - \lambda\int_{u=0}^{x}\bar{G}(u)B(x-u)f_n(u)du. 
\end{split}
\end{align}

\begin{align}
\begin{split}
I_2'(x) &= \lambda P_n(0) \int_{z=x}^{\infty}B(z-x)\lambda e^{-\lambda(z-x)} dz + \lambda \int_{z=x}^{\infty}\int_{u=0}^{z-x}\bar{G}(u)B(z-x-u)f_n(u)\lambda e^{-\lambda(z-x)} du dz \\
 & \qquad - P_n(0)\int_{z=x}^{\infty}f_{B}(z-x)\lambda e^{-\lambda(z-x)} dz - \int_{u=0}^{z-x}\bar{G}(u)f_{B}(z-x-u)f_n(u)\lambda e^{-\lambda(z-x)} du dz.
\end{split}
\end{align}

\n
Equation~\eqref{J(x)} yields

\begin{align*}
J(x) &=  \int_{0}^{\infty} dA(t)\int_{0}^{x+t} dF_{n}(u) G(u) - \int_{0}^{\infty} dA(t)\int_{0}^{t} dF_{n}(u) G(u)\\
     &= \int_{z=x}^{\infty}\int_{u=0}^{z} G(u)f_n(u)\lambda e^{-\lambda(z-x)}  du dz - \int_{z=x}^{\infty}\int_{u=0}^{z-x} G(u)f_n(u)\lambda e^{-\lambda(z-x)}  du dz \\
     &=: J_1(x) + J_{2}(x).
\end{align*}

\n
Differentiating with respect to $x$ the above equation, we get 

\begin{align}
J_{1}'(x) = \lambda \int_{z=x}^{\infty}\int_{u=0}^{z} G(u)f_n(u)\lambda e^{-\lambda(z-x)}  du dz - \lambda \int_{0}^{x}G(u)f_n(u) du.
\end{align}

\begin{align}
J_{2}'(x) = \lambda \int_{z=x}^{\infty}\int_{u=0}^{z-x} G(u)f_n(u)\lambda e^{-\lambda(z-x)}  du dz - \int_{z=x}^{\infty}G(z-x)f_n(z-x)\lambda e^{-\lambda(z-x)}  dz.
\end{align}

\n
Equation~\eqref{K(x)} can then  be rewritten

\begin{align*}
\begin{split}
K(x)  & = P_n(0)\int_{z=x}^{\infty}\int_{s=0}^{z-x} V(z-s)f_{B}(s) \lambda e^{-\lambda(z-x)} ds dz \\
      & \qquad + \int_{z=x}^{\infty}\int_{u=0}^{z-x}\int_{0}^{z-x-u}\bar{G}(u)V(z-u-s)f_{B}(s)f_n(u)\lambda e^{-\lambda(z-x)}  ds du dz\\
      &=: K_{1}(x) + K_{2}(x).
\end{split}
\end{align*}

\n \\
The derivative of $K_{1}$ and $K_{2}$ are given by

\begin{align}
\begin{split}
 K_{1}'(x) &= \lambda  P_n(0)\int_{z=x}^{\infty}\int_{s=0}^{z-x} V(z-s)f_{B}(s) \lambda e^{-\lambda(z-x)} ds dz \\
 & \qquad - P_n(0)\int_{z=x}^{\infty} f_{B}(z-x) \lambda e^{-\lambda(z-x)} dz.
\end{split}
\end{align}

\begin{align}
\begin{split}
K_{2}'(x) &= \lambda\int_{z=x}^{\infty}\int_{u=0}^{z-x}\int_{0}^{z-x-u}\bar{G}(u)V(z-u-s)f_{B}(s)f_n(u)\lambda e^{-\lambda(z-x)}  ds du dz\\
         & \qquad - V(x)\int_{z=x}^{\infty}\int_{u=0}^{z-x}\bar{G}(u)f_{B}(z-x-u)f_n(u)\lambda e^{-\lambda(z-x)} du dz.
\end{split}
\end{align}

\n
The last equation, \eqref{L(x)} yields

\begin{equation*}
L(x)  = \int_{z=x}^{\infty}\int_{u=0}^{z-x}G(u)V(z-x)f_n(u)\lambda e^{-\lambda(z-x)}du dz.
\end{equation*}

\n
The derivative is

\begin{align}
\begin{split}
L'(x) &= \lambda \int_{z=x}^{\infty}\int_{u=0}^{z-x}G(u)V(z-x)f_n(u)\lambda e^{-\lambda(z-x)}du dz \\
      & \qquad - V(x)\int_{x}^{\infty}G(z-x)f_n(z-x)\lambda e^{-\lambda(z-x)}dz.
\end{split}
\end{align}

\n \\
For $x >0$, differentiating both sides of Equation~\eqref{integral equation F(n+1)} whenever it exists, yields

\begin{align}
\label{eq f_n+1}
\begin{split}
f_{n+1}(x) & = I_1'(x) + I_2'(x) + J_1'(x) + J_{2}'(x)+K_{1}'(x)  + K_{2}'(x)+L'(x) \\
  & = \lambda  F_{n+1}(x) -\lambda P_n(0)B(x) - \lambda\int_{u=0}^{x} \bar{G}(u)B(x-u)f_n(u) du - \lambda\int_{u=0}^{x}G(u)f_n(u) du\\
  & \qquad  + P_n(0) \left[ 1-V(x) \right]\int_{z=x}^{\infty}f_{B}(z-x)\lambda e^{-\lambda(z-x)}dz \\
  & \qquad + \left[ 1-V(x) \right] \int_{z=x}^{\infty}\int_{u=0}^{z-x} \bar{G}(u)f_{B}(z-x-u)f_n(u)\lambda e^{-\lambda(z-x)}du dz\\
  & \qquad +  \left[ 1-V(x) \right] \int_{z=x}^{\infty} G(z-x)f_n(z-x) \lambda e^{-\lambda(z-x)}dz.
\end{split}           
\end{align}

\n
By definition,

\begin{equation*}
  F_{n+1}(x) = P_{n+1}(0) + \int_{u=0}^{x}f_{n+1}(u) du, \quad x > 0.
\end{equation*}

\n
Substituting the previous equality into~\eqref{eq f_n+1} and letting $n \rightarrow \infty$ gives the desired equation for the steady state pdf.

\end{proof}

\section{Proof of Lemma~\ref{chap2-proba:lemma1 tail asymptotic:limitinf}}

\begin{proof}
Using Equation~\eqref{integral equation of F-bar wrt dF_gamma}, we have for all $x \geq 0$
\begin{equation*}
\bar{F}(x) \geq  \rho_{2}\bar{V}^{r}(x) + \rho_{1}f^{*}(\gamma)\bar{F_{\gamma}}(x) + \rho_{1}f^{*}(\gamma)\bar{B}^{r}(x)F_{\gamma}(x).
\end{equation*}

\n Hence
\begin{equation*}
\dfrac{\bar{F}(x)}{\bar{B}^{r}(x)} \geq  \rho_{2}\dfrac{\bar{V}^{r}(x)}{\bar{B}^{r}(x)} + \rho_{1}f^{*}(\gamma)\dfrac{\bar{F_{\gamma}}(x)}{\bar{B}^{r}(x)} +\rho_{1}f^{*}(\gamma)F_{\gamma}(x).
\end{equation*}

\n\\ Since $\dfrac{\bar{F}_{\gamma}(x)}{\bar{B}^{r}(x)} \geq 0$ for all $x$, using assumption \textbf{H.1} and the fact that $F_{\gamma}$ is a distribution function, we have
\begin{equation*}
\liminf_{x \to \infty} \dfrac{\bar{F}(x)}{\bar{B}^{r}(x)} \geq \rho_{1}f^{*}(\gamma).
\end{equation*}
\end{proof}

\section{Proof of Lemma~\ref{chap2-proba:lemma2 tail asymptotic:limitsup}}

\begin{proof}

Let be $\varepsilon > 0$. Suppose $B^r \in \mathcal{L}$ holds and fulfilles condition~\eqref{regularity condition smith72}. According to Lemma 2.2 in \cite{Smith1972}, we can find $\Delta$ sufficiently large, such that
\begin{equation}
\label{application of lemma 22 smith72 case B^r longtail}
\underset{x \rightarrow \infty}{\limsup}\, \dfrac{1}{\bar{B}^r(x)}\int_{\Delta}^{x-\Delta} \bar{B}^r(x-u)\, dB^r(u) \leq \varepsilon.
\end{equation}

\n We fix $\Delta > 0$ chosen previously. We have from \eqref{integral equation of F-bar wrt dB^r}, for $x \geq 0$,
\begin{align*}
\label{F-bar splitted}
\begin{split}
\bar{F}(x) &= \rho_{1}f^*(\gamma) \bar{B}^{r}(x) + \rho_{2}\bar{V}^{r}(x)+\rho_{1}f^{*}(\gamma)\bar{F}_{\gamma}(x) \\
& \quad  + \rho_{1}f^{*}(\gamma)\left\lbrace \int_{0}^{\Delta}\bar{F}(x-u)dB^r(u) 
 + \int_{\Delta}^{x-\Delta}\bar{F}(x-u)dB^r(u) + \int_{x-\Delta}^{x}\bar{F}(x-u)dB^r(u)\right\rbrace,
\end{split}
\end{align*}

\n hence, for $\quad x \geq 0$,

\begin{align*}
\begin{split}
\bar{F}(x) &\leq \rho_{1}\left[ F(0) + f^{*}(\gamma) \right]\bar{B}^{r}(x) + \rho_{2}\bar{V}^{r}(x) + \rho_{1}f^{*}(\gamma) \bar{F}_{\gamma}(x-\Delta)B^{r}(\Delta)+\bar{F}_{\gamma}(0)\left[ \bar{B}^{r}(x-\Delta) - \bar{B}^{r}(x) \right]\\
& \qquad + \rho_{1}f^{*}(\gamma)\int_{\Delta}^{x-\Delta}\bar{F}_{\gamma}(x-u)dB^{r}(u).
\end{split}
\end{align*}

\n We now divide the above equation by $\bar{B}^r(x)$
\begin{align*}
\begin{split}
\dfrac{\bar{F}(x)}{\bar{B}^{r}(x)} & \leq \rho_{1}f^{*}(\gamma) + \rho_{2}\dfrac{\bar{V}^{r}(x)}{\bar{B}^{r}(x)} + \rho_{1}f^{*}(\gamma)\dfrac{\bar{F}_{\gamma}(x-\Delta)}{\bar{B}^{r}(x)} B^{r}(\Delta) + \dfrac{\rho_1 f^*(\gamma)}{\bar{B}^{r}(x)}\left[ \bar{B}^{r}(x-\Delta) - \bar{B}^{r}(x) \right]\\
& \qquad + \rho_{1}f^{*}(\gamma)\dfrac{1}{\bar{B}^{r}(x)}\int_{\Delta}^{x-\Delta}\bar{F}_{\gamma}(x-u)dB^{r}(u) \\
   & \leq \rho_{1}f^{*}(\gamma)  + \rho_{2}\dfrac{\bar{V}^{r}(x)}{\bar{B}^{r}(x)} + \rho_{1}f^{*}(\gamma)\dfrac{\bar{F}_{\gamma}(x-\Delta)}{\bar{B}^{r}(x-\Delta)}\dfrac{\bar{B}^{r}(x-\Delta)}{\bar{B}^{r}(x)} B^{r}(\Delta) \\
   & \qquad +  \dfrac{\rho_1 f^*(\gamma)}{\bar{B}^{r}(x)}\left[ \bar{B}^{r}(x-\Delta) - \bar{B}^{r}(x) \right] + \rho_{1}f^{*}(\gamma)\int_{\Delta}^{x-\Delta}\dfrac{\bar{F}_{\gamma}(x-u)}{\bar{B}^{r}(x-\Delta)}\dfrac{\bar{B}^{r}(x-\Delta)}{\bar{B}^{r}(x)}dB^{r}(u).
\end{split}
\end{align*}

\n Besides, we have
\begin{equation*}
\dfrac{\bar{F}_{\gamma}(x-\Delta)}{\bar{B}^{r}(x-\Delta)} \leq \dfrac{e^{-\gamma(x-\Delta)}}{f^{*}(\gamma)}\dfrac{\bar{F}(x-\Delta)}{\bar{B}^{r}(x-\Delta)},
\end{equation*}

\n and
\begin{align*}
\int_{\Delta}^{x-\Delta}\dfrac{\bar{F}_{\gamma}(x-u)}{\bar{B}^{r}(x-\Delta)}\dfrac{\bar{B}^{r}(x-\Delta)}{\bar{B}^{r}(x)}dB^{r}(u) &\leq \dfrac{e^{-\gamma \Delta}}{f^{*}(\gamma)}\int_{\Delta}^{x-\Delta}\dfrac{\bar{F}(x-u)}{\bar{B}^{r}(x-\Delta)}\dfrac{\bar{B}^{r}(x-\Delta)}{\bar{B}^{r}(x)}dB^{r}(u)\\
   &  \leq \dfrac{e^{-\gamma \Delta}}{f^{*}(\gamma)} M \int_{\Delta}^{x-\Delta}\dfrac{\bar{B}^{r}(x-\Delta)}{\bar{B}^{r}(x)}dB^{r}(u).
\end{align*}

\n As a consequence
\begin{align*}
\begin{split}
\dfrac{\bar{F}(x)}{\bar{B}^{r}(x)} & \leq \rho_{1}f^{*}(\gamma) + \rho_{2}\dfrac{\bar{V}^{r}(x)}{\bar{B}^{r}(x)} + \rho_{1}e^{-\gamma(x-\Delta)}\dfrac{\bar{F}(x-\Delta)}{\bar{B}^{r}(x-\Delta)}\dfrac{\bar{B}^{r}(x-\Delta)}{\bar{B}^{r}(x)} B^{r}(\Delta) \\
   & \qquad +  \dfrac{\rho_1 f^*(\gamma)}{\bar{B}^{r}(x)}\left[ \bar{B}^{r}(x-\Delta) - \bar{B}^{r}(x) \right] + \rho_{1}e^{-\gamma \Delta} M \int_{\Delta}^{x-\Delta}\dfrac{\bar{B}^{r}(x-\Delta)}{\bar{B}^{r}(x)}dB^{r}(u).
\end{split}
\end{align*}

Using the fact that $B^r(x-\Delta) \sim B^r(x)$, as $x \rightarrow \infty$, and $\underset{x \rightarrow \infty}{\limsup}\dfrac{\bar{F}(x)}{\bar{B}^r(x)} < \infty$ (by \textbf{H.3}), with equation~\eqref{application of lemma 22 smith72 case B^r longtail} and assumption \textbf{H.1}, we find
\begin{equation*}
\underset{x \to \infty}{\limsup} \dfrac{\bar{F}(x)}{\bar{B}^{r}(x)} \leq \rho_{1}f^{*}(\gamma) +  \rho_{1}e^{-\gamma \Delta} M \varepsilon.
\end{equation*}

\n Letting $\Delta \to \infty$, we get
\begin{equation*}
\underset{x \to \infty}{\limsup} \dfrac{\bar{F}(x)}{\bar{B}^{r}(x)}  \leq \rho_{1}f^{*}(\gamma).
\end{equation*}
\end{proof}

\section{Proof of Lemma~\ref{chap2-proba:lemma3 tail asymptotic:limitinf with F longtail}}

\begin{proof}

Let be $\varepsilon > 0$. Suppose $F \in \mathcal{L}$ holds and fullfiles condition~\eqref{regularity condition smith72}. According to Lemma 2.2 in \cite{Smith1972}, we can find a large $\Delta(\varepsilon)$ such that
\begin{equation}
\label{application of lemma 22 smith72 case F longtail}
\underset{x \rightarrow \infty}{\limsup}\, \dfrac{1}{\bar{F}(x)}\int_{\Delta(\varepsilon)}^{x-\Delta(\varepsilon)} \bar{F}(x-u)\, dF(u) \leq \varepsilon.
\end{equation}

From Lemma~\ref{chap2-proba:lemma1 tail asymptotic:limitinf}, there exist a small $\tilde{\varepsilon} > 0$ and $\Delta(\tilde{\varepsilon} )$ such that
\begin{equation}
\label{upper bound for the inverse B/F}
\dfrac{\bar{B}^r(x)}{\bar{F}(x)} \leq \rho_1 f^*(\gamma) + \tilde{\varepsilon}, \qquad x \geq \Delta(\tilde{\varepsilon} ).
\end{equation}

\n Define $\Delta := \max \left( \Delta(\varepsilon), \Delta(\tilde{\varepsilon} ) \right)$. For this fixed $\Delta> 0$ and from Equation~\eqref{integral equation of F-bar wrt dF_gamma}, we have
\begin{align*}
\label{chap2-proba:lemma3:F-bar splitted}
\begin{split}
\bar{F}(x) &=  \rho_{2}\bar{V}^{r}(x)+\rho_{1}f^{*}(\gamma)\bar{F}_{\gamma}(x) \\
& \quad  + \rho_{1}f^{*}(\gamma)\left\lbrace \int_{0^-}^{\Delta}\bar{B}^{r}(x-u)dF_{\gamma}(u) 
 + \int_{\Delta}^{x-\Delta}\bar{B}^{r}(x-u)dF_{\gamma}(u) + \int_{x-\Delta}^{x}\bar{B}^{r}(x-u)dF_{\gamma}(u)\right\rbrace.
\end{split}
\end{align*}

\n Besides, we have
\begin{equation*}
 \int_{0^-}^{\Delta}\bar{B}^{r}(x-u)dF_{\gamma}(u) \leq \bar{B}^{r}(x-\Delta)F_{\gamma}(\Delta),
\end{equation*}

\n and
\begin{equation*}
\int_{x-\Delta}^{x}\bar{B}^{r}(x-u)dF_{\gamma}(u) \leq \bar{F}_{\gamma}(x-\Delta) - \bar{F}_{\gamma}(x).
\end{equation*}

Then, 
\begin{align*}
\begin{split}
\bar{F}(x) &\leq  \rho_{2}\bar{V}^{r}(x)+\rho_{1}f^{*}(\gamma)\bar{F}_{\gamma}(x) + \rho_{1}f^{*}(\gamma)\bar{B}^{r}(x-\Delta)F_{\gamma}(\Delta) \\
&\qquad + \rho_{1}f^{*}(\gamma) \int_{\Delta}^{x-\Delta}\bar{B}^{r}(x-u)dF_{\gamma}(u) + \rho_{1}f^{*}\left[\bar{F}_{\gamma}(x-\Delta) - \bar{F}_{\gamma}(x) \right].
\end{split}
\end{align*}

\n Dividing the above equation by $\bar{F}$, we get
\begin{align}
\label{upper bound F-bar splitted afer make division by F-bar}
\begin{split}
1 \leq & \rho_{2}\dfrac{\bar{V}^{r}(x)}{\bar{B}^{r}(x)}+\rho_{1}f^{*}(\gamma)\dfrac{\bar{F}_{\gamma}(x)}{\bar{F}(x)} + \rho_{1}f^{*}(\gamma)\dfrac{\bar{B}^{r}(x-\Delta)}{\bar{F}(x-\Delta)}\dfrac{\bar{F}(x-\Delta)}{\bar{F}(x)}F_{\gamma}(\Delta) \\
&\qquad + \rho_{1}f^{*}(\gamma)\dfrac{1}{\bar{F}(x)} \int_{\Delta}^{x-\Delta}\bar{B}^{r}(x-u)dF_{\gamma}(u) + \rho_{1}f^{*}(\gamma)\dfrac{1}{\bar{F}(x)}\left[\bar{F}_{\gamma}(x-\Delta) - \bar{F}_{\gamma}(x) \right].
\end{split}
\end{align}

Let us remark that, $\dfrac{\bar{F}_{\gamma}(x)}{\bar{F}(x)} \leq \dfrac{e^{-\gamma x}}{f^{*}(\gamma)} $ for all $x \geq 0$, and for sufficiently large $x$
\begin{align*}
\dfrac{1}{\bar{F}(x)}\int_{\Delta}^{x-\Delta} \bar{B}^r(x-u)\, dF_{\gamma}(u) & \leq \dfrac{e^{-\gamma \Delta}}{f^*(\gamma)} \dfrac{1}{\bar{F}(x)}\int_{\Delta}^{x-\Delta} \bar{B}^r(x-u)\, dF(u)\\
&\leq \dfrac{e^{-\gamma \Delta}}{f^*(\gamma)} \dfrac{1}{\bar{F}(x)}\int_{\Delta}^{x-\Delta} \dfrac{\bar{B}^r(x-u)}{\bar{F}(x-u)} \bar{F}(x-u)\, dF(u)\\
& \leq \dfrac{e^{-\gamma \Delta}}{f^*(\gamma)} \left( \rho_1 f^*(\gamma) + \tilde{\varepsilon} \right) \dfrac{1}{\bar{F}(x)}\int_{\Delta}^{x-\Delta} \bar{F}(x-u) \, dF(u),
\end{align*}

\n  where the last inequality is due to \eqref{upper bound for the inverse B/F}. From~\eqref{upper bound F-bar splitted afer make division by F-bar} and for sufficiently large $x$

\begin{align*}
\begin{split}
1 &\leq \underset{x \rightarrow \infty}{\liminf} \left\lbrace  \rho_{2}\dfrac{\bar{V}^{r}(x)}{\bar{F}(x)}+\rho_{1}f^{*}(\gamma)\dfrac{\bar{F}_{\gamma}(x)}{\bar{F}(x)} + \rho_{1}f^{*}(\gamma)\dfrac{\bar{B}^{r}(x-\Delta)}{\bar{F}(x-\Delta)}\dfrac{\bar{F}(x-\Delta)}{\bar{F}(x)}F_{\gamma}(\Delta) \right\rbrace \\
&\qquad + \rho_1 e^{-\gamma \Delta}\left(\rho_1 f^{*}(\gamma)+ \tilde{\varepsilon}  \right) \underset{x \rightarrow \infty}{\limsup} \left\lbrace \dfrac{1}{\bar{F}(x)}\int_{\Delta}^{x-\Delta}\bar{F}(x-u)dF(u) \right\rbrace.
\end{split}
\end{align*}

\n
Thus, using the fact that $\bar{F}(x-\Delta) \sim \bar{F}(x)$ and assumption \textbf{H.2}, we find
\begin{equation*}
1 \leq \rho_{1}f^{*}(\gamma)F_{\gamma}(\Delta) \underset{x \to \infty}{\liminf} \dfrac{\bar{B}^{r}(x)}{\bar{F}(x)} + \rho_1 e^{-\gamma \Delta} \left(\rho_1 f^{*}(\gamma) + \tilde{\varepsilon}\right)\varepsilon.
\end{equation*}

\n Taking the limit as $\Delta \to \infty$, we obtain

\begin{equation*}
\underset{x \to \infty}{\liminf} \dfrac{\bar{B}^{r}(x)}{\bar{F}(x)} \geq \dfrac{1}{\rho_1 f^*(\gamma)},
\end{equation*}

\n since $F_{\gamma}(\Delta)\rightarrow 1$, as $\Delta \rightarrow \infty$, and leads to
\begin{equation*}
\underset{x \to \infty}{\limsup} \dfrac{\bar{F}(x)}{\bar{B}^{r}(x)} \leq \rho_{1}f^{*}(\gamma).
\end{equation*}
\end{proof}

\section{Proof of Theorem~\ref{chap2-proba:theorem asymptotic equivalence}}

\begin{proof}

Assume $B^r \in \mathcal{L}$ holds and satisfies condition~\eqref{regularity condition smith72}. According to Lemma~\ref{chap2-proba:lemma1 tail asymptotic:limitinf} and Lemma~\ref{chap2-proba:lemma2 tail asymptotic:limitsup}, we have
\begin{equation}
\label{chap2-proba: theorem limit ratio F-bar and B-bar}
\underset{x \rightarrow \infty}{\lim}\dfrac{\bar{F}(x)}{\bar{B}^r(x)} = \rho_1 f^*(\gamma).
\end{equation}

\n Let $y > 0$. It follows from the above limit that
\begin{equation*}
\underset{x \rightarrow \infty}{\lim}\dfrac{\bar{F}(x)}{\bar{F}(x-y)} \dfrac{\bar{F}(x-y)}{\bar{B}^r(x-y)} \dfrac{\bar{B}^r(x-y)}{\bar{B}^r(x)} = \rho_1 f^*(\gamma),
\end{equation*}
which yields $\bar{F}(x) \sim \bar{F}(x-y)$, as  $x \rightarrow \infty$. \vspace{0.2cm}

Conversely, if $F \in \mathcal{L}$ such that condition~\eqref{regularity condition smith72} holds, then by Lemma~\ref{chap2-proba:lemma1 tail asymptotic:limitinf} and Lemma~\ref{chap2-proba:lemma2 tail asymptotic:limitsup}, we get \eqref{chap2-proba: theorem limit ratio F-bar and B-bar}, which yields for fixed $y > 0$, $\bar{B}^r(x-y) \sim \bar{B}^r(x)$, as  $x \rightarrow \infty$.

\end{proof}

\end{appendices}

\bibliographystyle{plainnat}
\bibliography{biblio-general-queue}

\end{document}